\documentclass[onefignum,onetabnum]{siamart171218}

\usepackage{amsopn}

\usepackage{amssymb,latexsym,amsmath,enumerate,verbatim,amsfonts,mathdots,mathabx}
\usepackage{graphicx}
\usepackage{subcaption}
\usepackage{algorithm}
\usepackage{algorithmic}

\usepackage{color}
\usepackage{multirow}
\usepackage{caption}
\usepackage{tikz}
\usepackage{bm}
\usepackage{comment}
\usepackage{xcolor}

\usepackage{lipsum}
\usepackage{epstopdf}
\usepackage{microtype} 
\usepackage{listings} 
\definecolor{navyblue}{rgb}{0, 0, 0.5}

\def\ker{{\rm ker}}
\def\Range{{\rm Range}}

\def\H{{\mathcal{H}}}
\def\A{{\mathcal{A}}}
\def\L{{\mathcal{L}}}

\def\P{{\mathcal{P}}}

\def\dist{{\rm dist}}
\def\R{{\rm I\!R}}
\def\rank{{\rm rank}}
\def\dom{{\rm dom}}

\def\argmin{\mathop{\rm arg\,min}}
\def\nn{{\nonumber}}

\ifpdf
  \DeclareGraphicsExtensions{.eps,.pdf,.png,.jpg}
\else
  \DeclareGraphicsExtensions{.eps}
\fi

\newsiamremark{remark}{Remark}
\newsiamremark{conjecture}{Conjecture}
\crefname{hypothesis}{Hypothesis}{Hypotheses}
\newsiamthm{claim}{Claim}

\headers{Hybrid method for problems with multiple rank constraints}{T. Liu, I. Markovsky, T. K. Pong and A. Takeda}

\title{A  hybrid penalty method for a class of optimization problems with multiple rank constraints}

\author{Tianxiang Liu\thanks{RIKEN Center for Advanced Intelligence Project, 1-4-1, Nihonbashi, Chuo-ku, Tokyo 103-0027, Japan
 (\email{tianxiang.liu@riken.jp}).}
  \and Ivan Markovsky\thanks{Department ELEC, Vrije Universiteit Brussel (VUB), Pleinlaan 2, 1050 Brussels, Belgium
 (\email{imarkovs@vub.ac.be}).}
  \and Ting Kei Pong\thanks{Department of Applied Mathematics, The Hong Kong Polytechnic University, Hong Kong. This author was supported partly by Hong Kong Research Grants Council PolyU153004/18p.
 (\email{tk.pong@polyu.edu.hk}).}
  \and Akiko Takeda\thanks{Department of Creative Informatics, Graduate School of Information Science and Technology, the University of Tokyo, Tokyo, Japan
 (\email{takeda@mist.i.u-tokyo.ac.jp}),
 RIKEN Center for Advanced Intelligence Project, 1-4-1, Nihonbashi, Chuo-ku, Tokyo 103-0027, Japan
 (\email{akiko.takeda@riken.jp}).}}

\begin{document}

\maketitle

\begin{abstract}
  In this paper, we consider the problem of minimizing a smooth objective over multiple rank constraints on Hankel-structured matrices. This kind of problems arises in system identification, system theory and signal processing, where the rank constraints are typically ``hard constraints". To solve these problems, we propose a hybrid penalty method that combines a penalty method with a post-processing scheme. Specifically, we solve the penalty subproblems until the penalty parameter reaches a given threshold, and then switch to a local alternating ``pseudo-projection'' method to further reduce constraint violation. Pseudo-projection is a generalization of the concept of projection. We show that a pseudo-projection onto a {\em single} low-rank Hankel-structured matrix constraint can be computed efficiently by existing softwares such as SLRA (\hyperlink{solver}{Markovsky and Usevich, 2014}), under mild assumptions. We also demonstrate how the penalty subproblems in the hybrid penalty method can be solved by pseudo-projection-based optimization methods, and then present some convergence results for our hybrid penalty method. Finally, the efficiency of our method is illustrated by numerical examples.
\end{abstract}

\begin{keywords}
Hankel-structure, system identification, hybrid penalty method, pseudo-projection.
\end{keywords}

\begin{AMS}
15B05, 90C30.
\end{AMS}

\section{Introduction}
Many data modeling problems can be posed and solved as \emph{structured low-rank approximation} problems, \emph{i.e.}, problems of approximating matrices by preserving the structure but reducing the rank \cite{Markovsky19}. The to-be-approximated matrices are constructed from data and the model's complexity is related to the rank of the approximation---the lower the rank, the simpler the model. However, the simpler the model is, the higher the approximation error is. One way to deal with this fundamental trade-off between model complexity and model accuracy is to solve a sequence of low-rank approximation problems with increasing bounds on the rank.

In static linear data modeling problems, \emph{i.e.}, models defined by linear algebraic equations, the data matrices are \emph{unstructured.} All spectral and Fr\"{o}benius norm optimal unstructured low-rank approximations can be obtained from truncation of the singular value decomposition \cite{GV96}. This result, known as the Eckart--Young--Mirsky theorem \cite{EY36}, is at the heart of dimensionality reduction methods in machine learning \cite{SC04}. Unstructured low-rank approximation is equivalent to the principal component analysis in statistics and the total least squares in numerical linear algebra~\cite{MV07}.

The object of system theory, control, and signal processing is dynamical models. In linear time-invariant data modeling problems, \emph{i.e.}, for models defined by linear constant-coefficient difference equations, the data matrix is \emph{Hankel structured} \cite{CDLP05, IUM14,Markovsky08,MU13}. To see this, consider a system defined by the equation
\vspace{-2mm}
\begin{equation*}
p_0 y(t) + p_1 y(t+1) + \cdots + p_sy(t+s) = 0, \ \  {\rm for}\ \ t = 1,\ldots,T-s.
\end{equation*}
By definition, the time series $y = [y(1),\ldots,y(T)]^\top\in\R^T$ is a trajectory of the system if
\vspace{-3mm}
\begin{equation*}
p \H_{s+1}(y) = 0,
\end{equation*}
where $p := [ p_0 \ p_1 \ \cdots \ p_s] \neq 0$ is the parameter vector of the system and
\begin{equation*}
\H_{s+1}(y) :=
\begin{bmatrix}
y(1) & y(2) & y(3) & \cdots &  y(T-s) \\
y(2) & y(3) &    \iddots          & & y(T-s+1) \\
y(3) &  \iddots            &       &              &  \vdots\\
\vdots              &               &      &        &  \\
y(s+1) & y(s+2)             &  \cdots                  &  & y(T)
\end{bmatrix}
\end{equation*}
is a Hankel matrix, constructed from the time series. Therefore, $\rank(\H_{s+1}(y)) \le s$. The resulting Hankel structured low-rank approximation problem does not admit an analytic solution in terms of the singular value decomposition. For this reason, numerous local optimization \cite{Markovsky14} as well as convex relaxation \cite{Fazel02} methods are proposed for solving it.

In this paper, we consider a generalization of the Hankel structured low-rank approximation problem to \emph{multiple rank constraints.} An application that motivates this generalization is the common dynamics estimation problem in multi-channel signal processing \cite{MFG18, MLT19, PDV06}. Modeling each channel separately requires an individual rank constraint of a Hankel matrix in the optimization problem. Imposing the assumption that the channels have common dynamics then leads to an additional (coupling) rank constraint. The problem of common dynamics estimation is closely related to the problem of approximate common factor computation of multiple polynomials in computer algebra \cite{KL98, UM17}. Specifically, we consider the following optimization problem with multiple rank constraints:
\begin{eqnarray}
\min_{y_1,\cdots, y_N\in\R^{n}} & &\ \ f(y)\nn\\
{\rm s.t.} & &\ \ \rank(\H_{r_i + 1}(y_i)) \le r_i, \ \ \ i = 1,\ldots, N,\label{multi_model}\\
& & \ \ \rank\left([\H_{r+1}(y_1)\ \H_{r+1}(y_2)\ \cdots\ \H_{r+1}(y_N)]\right) \le r,\nn
\end{eqnarray}
where $y = vec(y_1\cdots y_N)$ (see Section~\ref{notation} for notation), $r_i$ and $r$ are positive integers satisfying $r_i \le r \le \lfloor \frac{n-1}2\rfloor$ ($i = 1, \ldots, N$), and $f$ represents the loss function, which is nonnegative, level-bounded and smooth with Lipschitz continuous gradient. For example, $f(y) = \frac12\|y - \widebar{y}\|^2$, where $\widebar{y}\in\R^{Nn}$ is the noisy observation signal.

For constrained problems such as \eqref{multi_model} with smooth objectives, a classical solution method is the gradient projection algorithm, whose iterations require projections onto the feasible set. However, the coupling structure of the last constraint in \eqref{multi_model} makes projection onto the feasible set a challenging problem: indeed,  even the projection onto the set defined by {\em each} single constraint in \eqref{multi_model} does not admit a closed-form solution. Thus, variants of proximal gradient algorithms cannot be directly applied to solving \eqref{multi_model}. Fortunately, we can show that one can obtain a so-called ``pseudo-projection" (see Definition~\ref{def_ps}) onto the set defined by {\em each} single constraint by some existing solvers such as SLRA \cite{MK14}, under mild assumptions.

Motivated by this, we adopt a penalty approach and construct penalty subproblems whose feasible regions are either $\R^n$, or defined by either the first $N$ constraints or the last constraint in \eqref{multi_model}: the pseudo-projections are easy to compute in all these cases. We then propose an algorithm vNPG$_{\rm major}$ for the penalty subproblems, making explicit use of the difference-of-convex (DC) structure of the penalty functions. The algorithm vNPG$_{\rm major}$ is a variant of NPG$_{\rm major}$ in \cite[Algorithm~2]{LPT18} and is based on computing pseudo-projections, which can be done efficiently for the feasible region of the penalty subproblems.

While approximate solutions to \eqref{multi_model} can now be obtained by our penalty method, such solutions are typically {\em not} feasible for \eqref{multi_model}. This is not ideal for applications such as system identification in which solution feasibility is an important concern \cite{Markovsky08}. Even though constraint violation can theoretically be reduced via solving a sequence of penalty subproblems with increasing weights in the penalty functions, in practice this strategy results in high computational cost and numerical instability. To resolve this issue, we shift to a post-processing method after obtaining a moderately accurate solution by our penalty method. Specifically, starting from such a solution obtained from the penalty method, we apply an alternating pseudo-projection method, alternating between the set defined by the first $N$ constraints in \eqref{multi_model} and that defined by the last constraint there, to reduce constraint violation.


Our main contributions are highlighted as follows:
\begin{itemize}
\item We propose a hybrid penalty method (Algorithm~\ref{alg_hyb}) for solving \eqref{multi_model}: A penalty scheme allowing three different kinds of penalty subproblems, followed by an alternating pseudo-projection method for post-processing. An algorithm, vNPG$_{\rm major}$ (Algorithm~\ref{alg_npg}), is proposed for the penalty subproblems.
\item We prove some convergence results for the hybrid penalty method, including an error bound for the penalty method (Theorem~\ref{errbd_thm}) and the convergence rate for the alternating pseudo-projection method (Theorem~\ref{thm_ap}).
\item We demonstrate how a pseudo-projection can be obtained by the solver SLRA \cite{MK14} in Section~\ref{pseudo_sec}, under mild assumptions.
\end{itemize}

The rest of this paper is organized as follows. In Section~\ref{notation}, we introduce notation and some basic properties of Hankel operators. The hybrid penalty method and the corresponding convergence analysis are presented in Section~\ref{sec_hybrid}. In Section~\ref{sec_pseudo}, we demonstrate how to compute pseudo-projections. Numerical simulation results are presented in Section~\ref{sec_numerical}. Finally,  we give some concluding remarks in Section~\ref{sec_conclusion}.

\section{Notation and preliminaries}\label{notation}
Throughout this paper, we let $\R^n$ denote the $n$-dimensional Euclidean space and $\|\cdot\|$ denote the Euclidean norm induced by vector inner product $\langle\cdot,\cdot\rangle$.  For an $x\in\R^n$, we let $x(i)$ denote its $i$th entry. For vectors $y_1,\cdots, y_N\in\R^n$, we let $vec\left(y_1 \cdots\, y_N\right):=[y_1^\top \cdots\ y_N^\top]^\top \in\R^{Nn}$.
Given a matrix $A\in\R^{m\times n}$, we let $\|A\|_F$ denote its Fr\"{o}benius norm, $\|A\|_2$ denote its spectral norm, $A^\top$ denote its transpose and $A(i,j)$ denote its $(i,j)$th entry.
For $A$, $B\in\R^{m\times n}$, we denote the matrix inner product by $\langle A, B\rangle: = \sum_{i=1}^m\sum_{j=1}^nA(i,j)B(i,j)$. For a linear operator $\A$, we use $\A^*$, $\Range(\A)$ and $\ker(\A)$ to denote its adjoint, range and kernel, respectively.

For an extended-real-valued function $h: \R^n\rightarrow\R\cup\{\infty\}$, we say that $h$ is proper if $\dom\,h := \{x: h(x) < \infty\}\neq\emptyset$, and is closed if it is lower semi-continuous. Following \cite[Definition~8.3]{RW98}, for a proper closed function $h: \R^n\rightarrow\R\cup\{\infty\}$, the regular subdifferential of $h$ at $y\in {\rm dom}\,h$ is defined as
\vspace{-1 mm}
\begin{equation*}
\widehat\partial h(y):= \Bigg\{u:\;\liminf_{\begin{subarray}
  \ v\to  y\\
    v\neq y
\end{subarray}}\frac{h(v) - h(y) - {u}^\top(v - y)}{\| v - y\|}\ge 0\Bigg\},
\end{equation*}
and the (limiting) subdifferential of $h$ at $y\in {\rm dom}\,h$ is defined as
\vspace{-1 mm}
\begin{equation*}
\partial h(y):= \{u:\exists u^t \to  u, y^t \stackrel{h}{\to} y \, \mbox{with} \, u^t\in \widehat\partial h(y^t)\, \mbox{for each $t$}\},\vspace{-2 mm}
\end{equation*}
where $y^t \stackrel{h}{\to} y$ means both $h(y^t)\to h(y)$ and $y^t \to y$.
We say that $\bar y$ is a stationary point of $h$ if $0\in \partial h(\bar y)$. It is known from \cite[Theorem~10.1]{RW98} that any local minimizer of $h$ is a stationary point.

For a nonempty closed set $\Omega\subseteq\R^n$, we let $\delta_\Omega$ denote the indicator function of $\Omega$, which is zero in $\Omega$ and is infinity otherwise. 
The regular normal cone and (limiting) normal cone of $\Omega$ at $y\in \Omega$ are defined by $\widehat N_{\Omega}(y):= \widehat\partial \delta_\Omega(y)$ and $N_{\Omega}(y):= \partial \delta_\Omega(y)$ respectively. We use $\dist(x,\Omega)$ to denote the distance from an $x\in \R^n$ to $\Omega$ and $\P_\Omega(x)$ to denote the projection, {\em i.e.}, $\dist(x,\Omega) := \inf_{y\in \Omega}\|x - y\|$ and $\P_\Omega(x) := \argmin_{y\in \Omega}\|x - y\|$.
For a nonempty closed set $\Omega\subseteq\R^{m\times n}$, the distance from an $X\in \R^{m\times n}$ to $\Omega$ and its projection are defined with respect to the Fr\"{o}benius norm:
\begin{equation*}
\dist(X, \Omega):= \inf_{Y\in\Omega}\|X - Y\|_F\ \ {\rm and}\ \  \P_{\Omega}(X):= \argmin_{Y\in\Omega} \|X - Y\|_F.
\end{equation*}

We next recall the definition of prox-regular sets; see \cite[Exercise 13.31]{RW98}.

\begin{definition} [Prox-regular sets] A closed set $\Omega$ is prox-regular at $\widebar{x}\in\Omega$ for $\widebar{v}\in N_{\Omega}(\widebar{x})$ if there exist $\epsilon > 0$ and $\sigma \ge 0$ such that whenever $x\in\Omega$ and $v\in N_{\Omega}(x)$ with $\|x - \widebar{x}\| < \epsilon$ and $\|v - \widebar{v}\| < \epsilon$, it holds that
\begin{equation*}
\langle v, y - x\rangle \le \frac{\sigma}2\|y - x\|^2 \ {\rm for\ all}\ y\in\Omega\ {\rm with}\ \|y - \bar x\|  < \epsilon.
\end{equation*}
Furthermore, $\Omega$ is prox-regular at $\widebar{x}$ if it is prox-regular at $\widebar{x}$ for all $\widebar{v}\in N_{\Omega}(\widebar{x})$.
\end{definition}

We now define the notion of pseudo-projection, which will be used in our subsequent discussions.
\begin{definition} [Pseudo-projection] \label{def_ps} Let $\Omega\subseteq\R^n$ be a nonempty closed set, $u\in\Omega$ and $x\in\R^n$. The pseudo-projection $\P^s_{\Omega}(x; u)$ of $x$ onto $\Omega$ with respect to $u$ is the collection of all $y\in \Omega$ satisfying:
\begin{enumerate}[{\rm (a)}]
  \item {\bf (Stationarity)} $x - y \in N_{\Omega}(y)$; and
  \item {\bf (Function value improvement)} $\|y - x\| \le \|u - x\|$.
\end{enumerate}
\end{definition}
Notice that any element of the pseudo-projection is a stationary point of the corresponding projection problem, {\em i.e.}, it is a stationary point of the function $w\mapsto \frac12\|w - x\|^2 + \delta_\Omega(w)$. Also, each such element improves the function value of the corresponding projection problem relative to a given point $u\in\Omega$. Pseudo-projection onto a nonempty closed set is always nonempty: indeed, in view of \cite[Example~6.16]{RW98} and \cite[Proposition~6.5]{RW98}, we have $\P_\Omega(x)\subseteq \P^s_{\Omega}(x; u)$ for all $x\in \R^n$ and all $u\in \Omega$.


For notational simplicity, we define linear operators $\L_i:\R^{Nn}\to \R^{(r_i+1)\times(n-r_i)}$ ($i = 1, \ldots, N$) and $\L: \R^{Nn} \to \R^{(r+1)\times N(n-r)}$ as
\begin{equation}\label{simple_map}
\begin{split}
  \L_i(y)& := \H_{r_i + 1}(y_i),\ \ \   i = 1,\ldots,N,\\
  \L(y) &:= [\H_{r + 1}(y_1)\ \H_{r + 1}(y_2)\ \cdots \ \H_{r + 1}(y_N)],
\end{split}
\end{equation}
where $y = vec\left(y_1 \cdots\, y_N\right)\in \R^{Nn}$, and $r_i$ ($i=1,\ldots,N$) and $r$ are defined in \eqref{multi_model}. We now present some properties of the linear operators $\H_l(\cdot)$ and $\L^*$.

\begin{lemma}\label{conj_map}
For any $Y\in\R^{(r+1)\times(n-r)}$,
\begin{equation*}
\H_{r+1}^*(Y) = \bigg[ Y(1,1)\, \cdots\, \overbrace{\sum_{i+j=k+1}Y(i,j)}^{{\rm the}\ k{\rm th\ element}}\, \cdots\, Y(r+1,n-r)\bigg]^\top\in\R^n.
\end{equation*}
\end{lemma}

\begin{lemma}\label{lemma}
For any $W_i\in\R^{(r + 1)\times(n - r)}$, $i = 1, \ldots, N$,  it holds
\begin{equation*}
\L^*[W_1\ W_2\ \cdots\ W_N] = vec\left(\H_{r + 1}^*(W_1)\ \H_{r + 1}^*(W_2)\ \cdots\ \H_{r + 1}^*(W_N)\right).
\end{equation*}
\end{lemma}

\begin{proof}
Fix any $W_i\in\R^{(r + 1)\times(n - r)}$, $i = 1, \ldots, N$. According to the definition of adjoint, for any $y = vec\left(y_1\ \cdots\ y_N\right)\in\R^{Nn}$, we have
\begin{eqnarray*}
\left\langle\L^*[W_1\ W_2\ \cdots\ W_N] , y\right\rangle &=& \left\langle [W_1\ W_2\ \cdots\ W_N], \L\left(vec\left(y_1\ \cdots\ y_N\right)\right)\right\rangle\\
&=& \left\langle [W_1\ W_2\ \cdots\ W_N],  [\H_{r+1}(y_1)\ \H_{r+1}(y_2)\ \cdots\ \H_{r+1}(y_N)]\right\rangle\\
&=& \sum_{i=1}^N\left\langle W_i, \H_{r+1}(y_i)\right\rangle = \sum_{i=1}^N\left\langle\H_{r+1}^*(W_i), y_i\right\rangle.
\end{eqnarray*}
Then the conclusion follows from this and the arbitrariness of $y$. This completes the proof.
\end{proof}

\section{A hybrid penalty method}\label{sec_hybrid}
Notice that there are multiple rank constraints in \eqref{multi_model}, making it difficult to compute the projection onto the feasible set. To handle these constraints, one intuitive idea is to use a penalty method to ``reduce" the number of constraints. Specifically, we replace some or all constraints by penalty functions which consist of penalty parameters and measures of constraint violation. However, approximate solutions returned by penalty methods are typically not feasible for \eqref{multi_model}. Although we can theoretically reduce constraint violation by increasing the weights in the penalty functions when feasibility is important ({\em e.g.}, in applications such as system identification \cite{Markovsky08}), this strategy leads to high computational cost and numerical instability in practice. One way out would be to shift to a local refinement method after obtaining a moderately accurate solution by the penalty method.

Based on these intuitive ideas, our solution method will then consist of two stages: a penalty method, followed by a post-processing scheme. We will describe the penalty method in Section~\ref{sdcam_sec}, the post-processing scheme in Section~\ref{ap_sec} and the hybrid penalty method and its convergence analysis in Section~\ref{conv_sec}.

\subsection{Stage 1: A penalty method} \label{sdcam_sec}

To describe the penalty method, we first rewrite \eqref{multi_model} as follows, using notation in \eqref{simple_map}:
\begin{eqnarray*}
\min_{y \in\R^{Nn}} & &\ \ f(y) \\
  {\rm s.t.} & &\ \  \rank \left(\L_i(y)\right) \le r_i, \ \ i = 1,\ldots, N,\\
  & &\ \  \rank\left(\L(y)\right) \le r.
 \end{eqnarray*}
This can be further equivalently rewritten as
\begin{equation}\label{uni_form}
\min_{y \in\R^{Nn}}  \ F(y):= f(y) + \delta_{\Omega}(y)+ \sum_{i = 1}^k\delta_{C_i}\left(\A_i(y)\right),
\end{equation}
with three ways of  setting $k$, $\A_i$, $\Omega$ and $C_i$:
\begin{itemize}
\item \textbf {Variant I:} $k = 1$,  $\A_1 = \L$ and
\begin{equation*}
\Omega = \left\{ y:\, \rank\left(\L_i(y)\right) \le r_i, \ \ i = 1, \ldots, N\right\},\ \ C_1:= \{Y: \rank(Y)\le r\}.
\end{equation*}

\item \textbf{Variant II:} $k = N$, $\A_i = \L_i$ ($i = 1, \ldots, N$) and
\begin{equation*}
\Omega = \{y:\, \rank\left(\L(y)\right)\le r\}, \ \  C_i = \{Y: \rank(Y)\le r_i\}, \ \ i = 1, \ldots, N.
 \end{equation*}

\item \textbf{Variant III:} $k = N + 1$,  $\A_i = \L_i$ ($i = 1, \ldots, N$), $\A_{N + 1} = \L$ and
\begin{equation*}
\Omega = \R^{Nn}, \, C_i = \{Y\hspace{-1mm}: \rank(Y)\le r_i\}, \, i = 1, \ldots, N, \,  C_{N + 1}= \{Y\hspace{-1mm}: \rank(Y)\le r\}.
\end{equation*}
\end{itemize}
Notice that for the above three variants, the projection onto $C_i$ has a closed-form solution. On the other hand, while the projection onto $\Omega$ does not in general admit a closed-form solution, some kinds of stationary points of this projection problem can be approximately and efficiently obtained by some existing solvers such as SLRA \cite{MK14}, as we will show in Section~\ref{pseudo_sec}, under mild assumptions.

Now we are ready to describe our penalty method. We first replace the constraints $\A_i(y)\in C_i$ ($i = 1, \ldots, k$) in \eqref{uni_form} by a penalty for violating the constraints to obtain the auxiliary function
\begin{equation}\label{penalty_form}
F_{\lambda}(y) = f(y) + \delta_{\Omega}(y) + \sum_{i=1}^k\frac1{2\lambda}\dist^2(\A_i(y), C_i),
\end{equation}
where $\lambda > 0$ is the penalty parameter. Then we approximately minimize the auxiliary function $F_{\lambda}(y)$ and update $y$ while decreasing $\lambda$.

Note that each term of the penalty function in \eqref{penalty_form} can be written as the Moreau envelope of indicator function $\delta_{C_i}(\cdot)$. Using the DC decomposition of the Moreau envelope as in \cite[Equation 6]{LPT18}, we see that
\begin{eqnarray}
F_{\lambda}(y) &=& f(y) + \delta_{\Omega}(y) + \sum_{i=1}^k\inf_{Y_i}\left\{\delta_{C_i}(Y_i) + \frac1{2\lambda}\|Y_i - \A_i(y)\|_F^2\right\}\nn \\
&=& f(y) + \delta_{\Omega}(y) + \sum_{i = 1}^k\left\{\frac1{2\lambda}\|\A_i(y)\|_F^2 - \sup_{Y_i\in C_i}\left\{\frac1{\lambda}\left\langle\A_i^*(Y_i),y\right\rangle - \frac1{2\lambda}\|Y_i\|_F^2\right\} \right\} \nn\\
&=& \underbrace{f(y)  + \sum_{i=1}^k\frac1{2\lambda}\|\A_i(y)\|_F^2}_{h(y)} +\ \delta_{\Omega}(y) - \underbrace{\sum_{i=1}^k\sup_{Y_i\in C_i}\left\{\frac1{\lambda}\left\langle\A_i^*(Y_i),y\right\rangle - \frac1{2\lambda}\|Y_i\|_F^2\right\}}_{g(y)}, \label{sub_prob}
\end{eqnarray}
where $h$ is a smooth function and $g$  is  a convex function with $\sum_{i=1}^k\frac1{\lambda}\A_i^*\left(\P_{C_i}(\A_i(y))\right)\subseteq\partial{g(y)}$; see \cite[Equation 7]{LPT18}. Recall that the projection onto $C_i$ is easy to compute. Thus, for Variant III, in which $\Omega = \R^{Nn}$, $F_{\lambda}$ can be minimized via NPG$_{\rm major}$ in \cite[Algorithm 2]{LPT18}. However, for Variants I and II, the projection onto $\Omega$ is not easy to compute. Fortunately, one can obtain some kind of stationary points for the corresponding projection problems via specific solvers: as we shall see in Section~\ref{pseudo_sec}, such a point belongs to the set of pseudo-projection (see Definition~\ref{def_ps}) under mild assumptions. Thus, we propose a variant of NPG$_{\rm major}$ as Algorithm~\ref{alg_npg}, which we call vNPG$_{\rm major}$, where we replace the projection in the subproblem by pseudo-projection.


\begin{algorithm}[h]
\caption {vNPG$_{\rm major}$ for minimizing \eqref{sub_prob}}
\label{alg_npg}

\begin{algorithmic}
\STATE{{\color{navyblue}{\bf Step 0.}} Choose $y^0\in\Omega$, $L_{\max} > L_{\min} > 0$, $\tau > 1$, $c > 0$ and an integer $M \ge 0$. Set $l = 0$.}

\STATE{{\color{navyblue}{\bf Step 1.}} Pick any $\xi^l\in \sum_{i=1}^k\frac1{\lambda}\A_i^*\left(\P_{C_i}(\A_i(y^l))\right)$ and arbitrarily choose $L_l^0\in[ L_{\min}, L_{\max}]$. For $L_{l,i} = L_l^0\tau^i$, $i = 0,1,\ldots$ compute
\begin{equation}\label{slra}
u_i^l  \in \P^s_{\Omega}\left(y^l - \frac1{L_{l,i}}(\nabla h(y^l) - \xi^l); y^l\right)
\end{equation}
 until some $u_i^l$ satisfies
\begin{equation}\label{non_ls}
F_{\lambda}(u_i^l) \le \max_{[l-M]_+\le j\le l}F_{\lambda}(y^j) - \frac{c}2\|u_i^l - y^l\|^2.
\end{equation}}

\STATE{{\color{navyblue}{\bf Step 2.}} Let $\widebar{L}_l = L_{l,i}$, $y^{l+1} = u_i^l$ and $l = l + 1$. Go to Step 1 unless some stopping criterion is met.}
\vspace{2mm}

\end{algorithmic}
\end{algorithm}

The well-definedness of \eqref{non_ls}, \emph{i.e.}, whether the line-search loop terminates after a finite number of iterations, will be discussed in Section~\ref{conv_sec}.

\subsection{Stage 2: Post-processing scheme}\label{ap_sec}

After we obtain an approximate solution by the penalty method, we shift to a post-processing method. A natural and simple choice for post-processing is the alternating projection method. Let
\begin{equation}\label{omega_ap}
\begin{aligned}
\Omega_1: &= \left\{y\in\R^{Nn}:\, \rank\left(\L_i(y)\right) \le r_i, \ \ i = 1, \ldots, N\right\},\\
\Omega_2 :&=  \{y\in\R^{Nn}: \rank\left(\L(y)\right) \le r\}.
\end{aligned}
\end{equation}
In the classical alternating projection method, one has to find the global minimizers of the following problems in each iteration, for some $\widetilde y$.
\begin{eqnarray}\label{subp_ge}
\min_{y = vec\left(y_1 \cdots\, y_N\right)\in\R^{Nn}}  \ \frac12\|y - \widetilde{y}\|^2&&\ \ {\rm s.t.}\ \   \rank(\H_{r_i+1}(y_i)) \le r_i,   \ i = 1,\ldots, N. \label{sub_ge_1}\\
\min_{y = vec\left(y_1 \cdots\, y_N\right)\in\R^{Nn}} \ \ \frac12\|y - \widetilde{y}\|^2&&\ \ {\rm s.t.}\ \ \rank\left([\H_{r+1}(y_1)\cdots \H_{r+1}(y_N)]\right) \le r. \label{sub_ge_2}
\end{eqnarray}
However, these problems are in general difficult to solve globally. Fortunately, as mentioned in Section~\ref{sdcam_sec}, we can obtain some point in the set of pseudo-projection efficiently, under mild assumptions. Thus, we adopt the following {\em alternating pseudo-projection method} for post-processing: start at some $x^0\in \Omega_2$ and $z^0\in \Omega_1$, let
 \begin{equation}\label{alg_AP}
z^{t+1} \in \P^s_{\Omega_1}(x^t; z^t) \ \ \ {\rm and}\ \ \ x^{t+1} \in \P^s_{\Omega_2}(z^{t+1}; x^t)\ \ \ t = 0, 1, \ldots
 \end{equation}

\subsection{Hybrid penalty method for (\ref{multi_model}) and convergence analysis}\label{conv_sec}

The hybrid penalty method for solving \eqref{multi_model}, which consists of the penalty method discussed in Section~\ref{sdcam_sec} and the post-processing method discussed in Section~\ref{ap_sec}, is presented as Algorithm~\ref{alg_hyb}.


\begin{algorithm}[h]
\caption{A hybrid penalty method for \eqref{multi_model}}
\label{alg_hyb}
\vspace{2mm}

\begin{algorithmic}

\STATE{\color{navyblue} \textbf{Penalty method for \eqref{uni_form}}}
\vspace{1mm}

\STATE{{\bf Step 0.} Pick two sequences of positive numbers with $\epsilon_t\downarrow 0$ and $\lambda_{t} \downarrow 0$,  choose a $\widebar{\lambda} \ge 0$, $y^{\rm feas} \in\Omega\cap \bigcap_{i=1}^k \A_i^{-1}(C_i)$ and $y^0\in\Omega$. Set $t = 0$.}
\vspace{1mm}

\STATE{{\bf Step 1.} If $F_{\lambda_t}(y^t) \le F_{\lambda_t}(y^{\rm feas})$, set $y^{t,0} = y^t$. Else, set $y^{t,0} = y^{\rm feas}$.}
\vspace{1mm}

\STATE{{\bf Step 2.} Approximately minimize $F_{\lambda_t}$ by Algorithm~\ref{alg_npg}, starting at $y^{t,0}$ and terminating at $y^{t,l_t}$ when the following three conditions hold:
\begin{equation}\label{stop_cri}
\begin{aligned}
\|y^{t,l_t+1} - y^{t,l_t}\|\le \epsilon_t,\ \ \ \ F_{\lambda_t}(y^{t,l_t})\le F_{\lambda_t}(y^{t,0}),\qquad\qquad\qquad\qquad\\
\dist\bigg(0,\nabla f(y^{t,l_t})\!+\!N_{\Omega}(y^{t,l_t+1})\!\!+\!\!\sum_{i=1}^k\frac1{\lambda_t}\A_i^*\!\left(\A_i(y^{t,l_t}) - \P_{C_i}(\A_i(y^{t,l_t}))\right)\!\!\bigg)\le \epsilon_t.
\end{aligned}
\end{equation}}
\vspace{-4mm}

\STATE{{\bf Step 3.} Update $ y^{t+1} =  y^{t,l_t}$ and $t=t+1$. If $\lambda_t < \widebar{\lambda}$ and $\widebar{\lambda} > 0$ , go to Step 4; otherwise go to Step 1.}

\vspace{2mm}

\STATE{\color{navyblue} \textbf{Post-processing method for \eqref{omega_ap}}}
\vspace{1mm}

\STATE{{\bf Step 4.} Let $x^0 \in \P^s_{\Omega_2}(y^{t+1}; 0)$ and $z^0\in\P^s_{\Omega_1}(y^{t+1}; 0)$, use alternative pseudo-projection as follows until some termination criterion is met:
\vspace{-1mm}
 \begin{equation}\label{post_iter}
z^{t+1} \in \P^s_{\Omega_1}(x^t; z^t) \ \ \ {\rm and}\ \ \ x^{t+1} \in \P^s_{\Omega_2}(z^{t+1}; x^t)\ \ \ t = 0, 1, \ldots
 \end{equation}}
\end{algorithmic}
\end{algorithm}

For the rest of the section, we will analyze the convergence of the hybrid penalty method, including the convergence analysis for the penalty method in Section~\ref{sub_pen} and the convergence rate for the post-processing method in Section~\ref{sub_ap}. Before proceeding, we first show that the criteria \eqref{non_ls} and \eqref{stop_cri} are well-defined.

\subsubsection{Well-definedness of (\ref{non_ls}) and (\ref{stop_cri})}
The following theorem is about the well-definedness of the line-search criterion \eqref{non_ls} and the termination criterion \eqref{stop_cri}, \emph{i.e.}, they can be satisfied after finitely many number of inner iterations. The proof is similar to that in \cite[Proposition~1]{LPT18}.

\begin{theorem}\label{thm_prox}
The line-search criterion \eqref{non_ls} is well-defined. Moreover, $\{\widebar{L}_l\}$ is bounded. Furthermore, the termination criterion \eqref{stop_cri} for Algorithm~\ref{alg_npg} is well-defined.
\end{theorem}
\begin{proof}
We start by discussing the line-search criterion. First, we observe from \eqref{slra} and Definition~\ref{def_ps} that
\begin{equation*}
\left\|u_i^l - (y^l - \frac1{L_{l,i}}(\nabla h(y^l) - \xi^l))\right\|^2 \le \left\|y^l - (y^l - \frac1{L_{l,i}}(\nabla h(y^l) - \xi^l))\right\|^2,
\end{equation*}
which is equivalent to
\vspace{-2mm}
\begin{equation}\label{haha1}
\langle\nabla h(y^l) - \xi^l, u_i^l - y^l\rangle \le - \frac{L_{l,i}}2\|u_i^l - y^l\|^2.
\end{equation}
Next, recall from the definition of $\xi^l$ and \cite[Equation 7]{LPT18} that
\begin{equation}\label{haha2}
\xi^l\in \sum_{i=1}^k\frac1{\lambda}\A_i^*\left(\P_{C_i}(\A_i(y^l))\right)\subseteq \partial g(y^l).
\end{equation}
Using \eqref{haha1} and \eqref{haha2} together with $u_i^l\in\Omega$, the $L$-smoothness of $h$ and the convexity of $g$ gives (here, we let $L$ denote the Lipschitz continuity modulus of $\nabla h$):
\begin{equation*}
\begin{split}
&F_{\lambda}(u_i^l) = h(u_i^l)  - g(u_i^l) \le h(y^l) + \langle\nabla h(y^l),  u_i^l - y^l\rangle + \frac{L}2\|u_i^l - y^l\|^2 - g(u_i^l)\\
&\le h(y^l) + \langle\nabla h(y^l),  u_i^l - y^l\rangle + \frac{L}2\|u_i^l - y^l\|^2 - g(y^l) -\langle\xi^l, u_i^l - y^l\rangle\\
&= F_{\lambda}(y^l) + \langle\nabla h(y^l) - \xi^l, u_i^l - y^l\rangle + \frac{L}2\|u_i^l - y^l\|^2\le F_{\lambda}(y^l) + \frac{L - L_{l,i}}2\|u_i^l - y^l\|^2.
\end{split}
\end{equation*}
Thus, we see that \eqref{non_ls} is satisfied whenever $L_{l,i} \ge L + c$. From the definition of $L_{l,i}$, this latter inequality must hold when $i$ satisfies $\tau^i L_{\min} \ge L+c$, implying that the line-search criterion \eqref{non_ls} is well-defined. Now, the boundedness of $\{\bar L_t\}$ can be argued as in \cite[Proposition~1]{LPT18}.

Next, let $\{y^l\}$ be generated by Algorithm~\ref{alg_npg} starting at a $y^{t,0}$ in Step 2 of Algorithm~\ref{alg_hyb}. We show that the termination criteria \eqref{stop_cri} hold after finitely many iterations in Algorithm~\ref{alg_npg} (with $y^l$ in place of $y^{t,l_t}$ and $y^{l+1}$ in place of $y^{t,l_t+1}$ in \eqref{stop_cri}). First, from \eqref{non_ls}, it is easy to see that the second inequality in \eqref{stop_cri} holds. Moreover, using a similar line of arguments as in \cite[Lemma 4]{WNF09}, we can show that
\begin{equation}\label{needed}
\lim_{l\rightarrow\infty}\|y^{l+1} - y^l\| = 0.
\end{equation}
Thus, the first inequality in \eqref{stop_cri} also holds after a finite number of iterations in Algorithm~\ref{alg_npg}. Finally, we note from \eqref{slra} and Definition~\ref{def_ps} that
\begin{equation*}
y^l - \frac1{\widebar{L}_l}(\nabla h(y^l) - \xi^l) - y^{l+1} \in N_{\Omega}(y^{l+1}).
\end{equation*}
Using this together with the definition of $h$ in \eqref{sub_prob}, we further obtain
\begin{equation*}
\widebar{L}_l(y^l - y^{l+1}) - \nabla f(y^l) - \sum_{i=1}^k\frac1{\lambda}\A_i^*\A_i(y^l) + \xi^l\in N_{\Omega}(y^{l+1}).
\end{equation*}
Combining this relation with \eqref{haha2} gives
\begin{equation*}
\dist\bigg(0, \nabla f(y^l) + N_{\Omega}(y^{l+1}) + \sum_{i=1}^k\frac1{\lambda}\A_i^*\left(\A_i(y^l)- \P_{C_i}(\A_i(y^l))\right)\bigg) \le \widebar{L}_{l}\|y^{l+1}- y^{l}\|.
\end{equation*}
This inequality together with \eqref{needed} and the boundedness of $\{\bar L_l\}$ shows that the third inequality in \eqref{stop_cri} holds after a finite number of iterations.
This completes the proof.
\end{proof}

\subsubsection{Convergence analysis for the penalty method in Algorithm~\ref{alg_hyb}}\label{sub_pen}

Notice that when $\widebar{\lambda} = 0$, the penalty method in Algorithm~\ref{alg_hyb} is exactly the same as \cite[Algorithm~1]{LPT18}. Thus, we know from \cite[Theorem~2]{LPT18} that the sequence $\{y^t\}$ is bounded and that, under some constraint qualifications, any accumulation point of sequence $\{y^t\}$ is a stationary point of \eqref{uni_form}.

We next estimate the violation of the constraints for the solution given by the penalty method in Algorithm~\ref{alg_hyb} in the following theorem. It implies that the constraint violation can be suppressed by terminating the algorithm at a small $\lambda_t$.

\begin{theorem}\label{errbd_thm}
Let $\{y^t\}$ be the sequence generated  by the penalty method  in Algorithm~\ref{alg_hyb} for solving \eqref{uni_form}. Then we have for $t\ge 1$ and $i=1,...,k$ that
\begin{equation*}
   {\rm dist}\left(\A_i(y^{t}), C_i\right) \le \sqrt{2\lambda_{t-1}f(y^{{\rm feas}})}.
\end{equation*}
\end{theorem}

\begin{proof}
Note from the nonnegativity of $f$, the definition of $y^t$, the second inequality in \eqref{stop_cri} and the choice of $y^{t,0}$ and $y^{\rm feas}$ that for $i = 1, \ldots, k$,
\begin{equation*}
\begin{split}
 & \frac1{2\lambda_{t-1}}{\rm dist}^2(\A_i(y^{t}), C_i) \le F_{\lambda_{t-1}}(y^{t}) = F_{\lambda_{t-1}}(y^{t-1,l_{t-1}})\\
 & \le F_{\lambda_{t-1}}(y^{t-1,0}) \le F_{\lambda_{t-1}}(y^{\rm feas}) = f(y^{\rm feas}).
 \end{split}
\end{equation*}
This completes the proof.
\end{proof}

\subsubsection{Convergence analysis of the post-processing method in Algorithm~\ref{alg_hyb}}\label{sub_ap}

 First, we present the following theorem which will be used later for the convergence analysis of the post-processing method in Algorithm~\ref{alg_hyb}.
\begin{theorem}\label{prox_proof}
Let $\Omega_2$ be defined as in \eqref{omega_ap}. Then $\Omega_2$ is prox-regular at any $\widebar{y}\in\Omega_2$ that satisfies $\rank(\L(\widebar{y})) = r$.
\end{theorem}
\begin{proof}
First, we can rewrite $\Omega_2$ as
\begin{equation*}
\Omega_2 = \{y\in\R^{Nn}: \L(y)\in C\}\ \ {\rm with}\ \ C:= \{Y\in\R^{(r+1)\times{N(n-r)}}: \rank(Y)\le r\}.
\end{equation*}
By \cite[Corollary 2.3]{PR10}, we see that $\Omega_2$ is prox-regular at $\widebar{y}\in\Omega_2$ if the following conditions hold:
\begin{itemize}
\item[(a)] there is no $z\neq 0$ in $N_C(\L(\widebar{y}))$ with $\L^*z = 0$;

\item[(b)] for every $\widebar{v}\in N_{\Omega_2}(\widebar{y})$, the set $C$ is prox-regular at $\L(\widebar{y})$ for every $z\in N_{C}(\L(\widebar{y}))$ with $\L^*z = \widebar{v}$.
\end{itemize}
We will prove that the above two statements hold. First, we prove (a). Using $\rank(\L(\widebar{y})) = r$ and noting that by assumption, we have $r\le \frac{n-1}2$ and hence $N(n-r)\ge r+1$, we see from \cite[Proposition 3.6]{Luke13} that
\begin{equation}\label{nc_form}
N_{C}(\L(\widebar{y})) = \left\{W: \left[\ker(W)\right]^\perp\cap\left[\ker(\L(\widebar{y}))\right]^\perp = \{0\}\,\, {\rm and}\,\,\rank(W)\le 1\right\}.
\end{equation}
On the other hand, we see from Lemma~\ref{lemma} that for any $W = [W_1\ W_2\ \cdots\ W_N]$ with $W_{\ell}\in\R^{(r+1)\times(n-r)}$ ($\ell = 1, \ldots, N$), we have
\begin{equation}\label{w_map}
\L^*[W_1\ W_2\ \cdots\ W_N] = vec\left(\H_{r + 1}^*(W_1)\ \H_{r + 1}^*(W_2)\ \cdots\ \H_{r + 1}^*(W_N)\right).
\end{equation}
Suppose that there exists some $\widehat{W} = [\widehat{W}_1\ \cdots\ \widehat{W}_N]\in N_C(\L(\widebar{y}))\cap \ker(\L^*)$ with $\widehat{W}_{\ell}\in\R^{(r+1)\times(n-r)}$ ($\ell = 1, \ldots, N$).  We then know from \eqref{nc_form} and \eqref{w_map} that
\begin{equation}\label{two_eq}
\rank(\widehat{W})\le 1 \ \ {\rm and}\ \ \H_{r+1}^*(\widehat{W}_{\ell}) = 0 \ \ {\rm for\ all}\ \ell = 1,\ldots,N.
\end{equation}
Now we fix  any $\ell$. Note from \eqref{two_eq} and Lemma~\ref{conj_map} that
\begin{equation}\label{use}
\rank(\widehat{W}_{\ell}) \le 1,\ \ \ \ \sum_{i+j=k+1}\widehat{W}_{\ell}(i,j) = 0,\ \ {\rm for\ any}\ \ k = 1,\ldots, n.
\end{equation}
We claim that $\widehat{W}_{\ell} = 0$. To prove this, we establish the following equivalent statement: for each $k = 1, \ldots, n$, all elements in the following set equal 0:
\begin{equation*}
S_k := \left\{\widehat{W}_{\ell}(i,j): i+ j  = k + 1\right\}.
\end{equation*}
First, it is easy to see from the equality in \eqref{use} that all elements in $S_1$ and $S_n$ are zero. Now we prove that every element in $S_k$ is zero by induction for each $k=1,2,\ldots,n-1$.

Suppose that there exists some $K\ge 1$ so that every element in $\bigcup_{\ell=1}^KS_\ell$ is zero. Let $\widehat{W}_{\ell}(\widebar{i}, \widebar{j})$ and $\widehat{W}_{{\ell}}(\widehat{i}, \widehat{j})$ be any two elements in $S_{K+1}$ with $\widebar{i} < \widehat{i}$. We then know from the first inequality in \eqref{use} that the $2\times 2$ submatrix formed by $\widehat{W}_{\ell}(\widebar{i}, \widehat{j})$, $\widehat{W}_{\ell}(\widebar{i}, \widebar{j})$, $\widehat{W}_{\ell}(\widehat{i}, \widehat{j})$ and $\widehat{W}_{\ell}(\widehat{i}, \widebar{j})$ is singular. Since $\widebar{i} + \widehat{j} < \widehat{i} + \widehat{j} = K + 2$, we conclude that $\widehat{W}_{\ell}(\widebar{i},\widehat{j}) = 0$ by the induction hypothesis. Consequently, there is at least one 0 in $\{\widehat{W}_{\ell}(\widebar{i}, \widebar{j}),\widehat{W}_{{\ell}}(\widehat{i}, \widehat{j})\}$. By the arbitrariness of these two elements in $S_{K+1}$, we see that there is at most one nonzero element in $S_{K+1}$. This together with the equality in \eqref{use} implies that every element in $S_{K+1}$ equals 0. Thus, we have $\widehat{W}_{\ell} =0$ by induction. Since $\ell$ is arbitrary, we see further that $\widehat{W} = 0$. This proves that $N_C(\L(\widebar{y}))\cap \ker(\L^*) = \{0\}$, which is equivalent to statement (a).

Now we prove (b). Using $\rank(\L(\widebar{y})) = r$, we know from \cite[Proposition 3.8]{Luke13} that $C$ is prox-regular at $\L(\widebar{y})$. Then by the definition of prox-regularity, we see that (b) holds. This completes the proof.
\end{proof}

Since \eqref{post_iter} involves the pseudo-projection instead of the actual projection, the post-processing method in Algorithm~\ref{alg_hyb} is different from the classical alternating projection method. Nevertheless, we can still show that the post-processing method in Algorithm~\ref{alg_hyb} has local linear convergence under commonly used assumptions for establishing local linear convergence of the alternating projection method (see, for example, the assumptions used in \cite[Theorem~5.16]{LLM07} and \cite[Theorem~4.2]{Luke13}). The proof follows the same line of arguments as in \cite[Theorem~5.2]{LLM07}. We include the proof in the Appendix for the convenience of the readers.

\begin{theorem}\label{thm_ap}
Let $\Omega_1$ and $\Omega_2$ be defined as in \eqref{omega_ap} and suppose that there exists some $\widebar{y}\in\Omega_1\cap\Omega_2$ such that $\rank(\L(\widebar{y})) = r$ and $N_{\Omega_1}(\widebar{y})\cap -N_{\Omega_2}(\widebar{y})  = \{0\}$.
Then for any initial points $x^0\in\Omega_2$ and $z^0\in\Omega_1$ near $\widebar{y}$, any sequence generated by the following iterations converges to a point in $\Omega_1\cap\Omega_2$ $R$-linearly:
\begin{equation}\label{aap_iter}
z^{t+1} \in \P^s_{\Omega_1}(x^t; z^t) \ \ \ {\rm and}\ \ \ x^{t+1} \in \P^s_{\Omega_2}(z^{t+1}; x^t)\ \ \ t = 0, 1, \ldots
\end{equation}
\end{theorem}

\section{Subproblem: pseudo-projection}\label{pseudo_sec}\label{sec_pseudo}

In this section, we consider the pseudo-projection subproblems \eqref{slra} in Algorithm~\ref{alg_npg} and \eqref{post_iter} in Algorithm~\ref{alg_hyb}. Recall that their corresponding {\em projection} problems can be put in the following general form:
\begin{equation}\label{line_form}
\min_{y\in\R^d}\ \ \frac12\|y - \widehat{y}\|^2 \ \ {\rm s.t.}\ \ \rank(\A(y)) \le m;
\end{equation}
here, $\A(y)\in\R^{p\times q}$, and $d$, $m$, $p$, $q$ and $\A$ are given as in \eqref{case_1} or \eqref{case_2} below, corresponding to \eqref{sub_ge_1} and \eqref{sub_ge_2} respectively:
\begin{eqnarray}
& &\ \ d = n, \, m = r_i, \, p = r_i + 1, \,  q = n - r_i,\,  \A(y) = \H_{r_i+1}(y). \label{case_1}\\
& &\ \ d = Nn, \, m = r,  \, p = r+ 1,  \, q = N(n - r),\, \A(y) = [\H_{r+1}(y_1)\cdots \H_{r+1}(y_N)]. \label{case_2}
\end{eqnarray}
The pseudo-projection problem corresponding to \eqref{line_form} can now be stated as follows: given $\widehat{y}\in\R^d$ and some reference point $y_b\in\R^d$ satisfying $\rank(\A(y_b)) \le m$, compute
\begin{equation*}
y_s\in\P^s_{\{y:\; \rank(\A(y)) \le m\}}(\widehat{y}; y_b).
\end{equation*}

In what follows, we will describe how such a $y_s$ can be obtained by the solver SLRA in \cite{MK14}. Recall that SLRA was developed based on the following key observation:
\begin{equation*}
\rank(\A(y)) \le m \Longleftrightarrow \ \exists\ {\rm full\ row\hbox{-}rank\ matrix}\ R\in\R^{(p-m)\times p}\ {\rm such\ that}\ R\A(y) = 0.
\end{equation*}
In view of this, algorithms were developed in \cite{MK14} to approximately solve the following equivalent formulation of \eqref{line_form}:
\begin{equation}\label{R_prob}
\min_{R\in\R^{(p - m)\times p}}\ \  \Psi(R) \ \ {\rm s.t.}\ \ RR^\top = I,
\end{equation}
where
\begin{equation}\label{R_form}
\Psi(R) :=\inf_{y\in\R^d}\left\{\frac12\|y - \widehat{y}\|^2:\; R\A(y) = 0\right\}.
\end{equation}

Notice that under the settings in \eqref{case_1} or \eqref{case_2}, we have $p-m = 1$ and hence \eqref{R_prob} is an optimization problem in $\R^{1\times p}$ and the feasible set reduces to $\{R\in \R^{1\times p}:\; RR^T = 1\}$.
We will show below in Section~\ref{sec41} that $\Psi$ in \eqref{R_form} is smooth on $\R^{1\times p}\backslash\{0\}$. Thus, when gradient-based optimization methods such as those described in \cite{MK14} are applied to solving \eqref{R_prob}, one obtains a stationary point of the following function:
\begin{equation}\label{obj}
  \widetilde \Psi(R):=\Psi(R) + \delta_{\Theta}(R), \mbox{ where $\Theta := \{R\in\R^{1\times p}:\; RR^T = 1\}$}.
\end{equation}
We will then discuss in Section~\ref{sec42} how an element of $\P^s_{\{y:\; \rank(\A(y)) \le m\}}(\widehat{y}; y_b)$ can be obtained from such a stationary point under mild assumptions.

\subsection{Smoothness of $\Psi$}\label{sec41}

In this subsection, we will prove that $\Psi$ is smooth on $\R^{1\times p}\backslash\{0\}$. We start with an auxiliary lemma.

\begin{lemma}\label{lemma_LICQ}
Consider \eqref{line_form} with setting \eqref{case_1} or \eqref{case_2}. For any $U\in\R^{1\times q}$ and any $R\in\R^{1\times p}\backslash\{0\}$, if $\A^*({R}^\top U) = 0$, then $U = 0$.
\end{lemma}
\begin{proof}
Assume that $U\in\R^{1\times q}$ and $R\in\R^{1\times p}\backslash\{0\}$ satisfy $\A^*({R}^\top U) = 0$. We need to show that $U = 0$.

We first consider \eqref{line_form} with setting \eqref{case_1}. In this case, we have $m = r_i$, $p = r_i + 1$, $q = n - r_i$ and $\A(y) = \H_{r_i + 1}(y)$. Notice that $R^\top\in\R^{p\times1} = \R^{r_i+1}$ and $U^\top\in\R^{q\times1} = \R^{n-r_i}$. Write
\begin{equation*}
R = \left[\, R(1),\ldots, R(r_i+1)\, \right],\ \ \ \  U =  \left[\, U(1),\ldots, U(n-r_i)\,\right],
\end{equation*}
and $W = R^\top U$. Using Lemma~\ref{conj_map}, we obtain
\begin{small}
\begin{equation*}
\begin{split}
&\A^*(R^\top U) = \H_{r_i + 1}^*(W) = \bigg[\cdots \overbrace{\sum_{s+t = k+1} W(s,t)}^{{\rm the}\ k{\rm th\ element}} \cdots \bigg]^\top= \bigg[\cdots \overbrace{\sum_{s+t = k+1} R(s)U(t)}^{{\rm the}\ k{\rm th\ element}} \cdots \bigg]^\top\\
&= \hspace{-1mm}{\underbrace{\begin{bmatrix}
\hspace{-0.7mm}
R(1) & R(2) & \cdots & R(r_i+1) & & & & & \\
 & R(1) & \cdots & R(r_i) & R(r_i+1) & & & & \\
 & & \ddots  & \ddots & & \ddots & & & \\
 & & & & & & R(1) & \cdots & R(r_i+1)
\hspace{-0.7mm}
\end{bmatrix}^\top}_{\widehat{R}}}
\begin{bmatrix}
\hspace{-0.8mm}U(1)\\ U(2)  \\  \vdots\\ U(n-r_i)\hspace{-0.8mm}
\end{bmatrix}.
\end{split}
\end{equation*}
\end{small}

Since $\A^*({R}^\top U) = 0$, to show that $U = 0$, it suffices to show that the $\widehat{R}\in \R^{n\times(n-r_i)}$ above has full column rank.
To this end, we first note from $R\in \R^{1\times (r_i+1)}\backslash\{0\}$ that there is at least one nonzero element in $R$.
Let $\widebar{i}$ be the first integer in $1,\ldots, r_i+1$ with $R(\widebar{i})\neq 0$. Then the $(n-r_i)\times(n-r_i)$ submatrix of $\widehat{R}$ starting from the $\widebar{i}$th row is lower triangular with all diagonal entries being $R(\widebar{i})\neq 0$. Consequently, this submatrix is nonsingular and thus $\widehat{R}$ has full column rank.  This completes the proof for this case.

Now we consider \eqref{line_form} with setting \eqref{case_2}.  In this case, we have $m = r$, $p = r+ 1$, $q = N(n - r)$ and $\A(y) = \L(y) = [\H_{r+1}(y_1)\cdots \H_{r+1}(y_N)]$ with $y = vec(y_1\cdots y_N)$. Notice that $R^\top\in\R^{p\times1} = \R^{r+1}$ and $U^\top\in\R^{q\times1} = \R^{N(n-r)}$.
Write
\begin{equation*}
R = \left[R(1),\ldots, R(r+1)\right],\ \ \ \  U =  \left[U_1,\ldots, U_N\right],
\end{equation*}
where $U_i^\top\in\R^{n-r}$ ($i=1,\ldots, N$).  We then see from Lemma~\ref{lemma} that
\begin{equation*}
\A^*(R^\top U) = \L^*(R^\top U) = vec\bigg(\H_{r+1}^*(R^\top U_1)\cdots \overbrace{\H_{r+1}^*(R^\top U_k)}^{{\rm the}\ k{\rm th\ block}} \cdots \H_{r+1}^*(R^\top U_N)\bigg).
\end{equation*}
Similar to the proof in setting \eqref{case_1}, we can write the $k$th block of $\A^*(R^\top U)$ as
\begin{small}
\begin{equation*}
{\underbrace{\begin{bmatrix}
R(1) & R(2) & \cdots & R(r+1) & & & & & \\
 & R(1) & \cdots & R(r) & R(r+1) & & & & \\
 & & \ddots  & \ddots & & \ddots & & & \\
 & & & & & & R(1) & \cdots & R(r+1)
\end{bmatrix}^\top}_{\widebar{R}}}
\begin{bmatrix}
U_k(1)\\ U_k(2)  \\  \vdots\\ U_k(n-r)
\end{bmatrix}.
\end{equation*}
\end{small}Consequently, we have
\begin{equation}\label{last_blk}
\A^*({R}^\top U) = \begin{bmatrix}\widebar{R} & & \\ &\ddots &\\ & & \widebar{R}\end{bmatrix}U^\top.
\end{equation}
Since $\A^*({R}^\top U) = 0$, to prove that $U = 0$, we only need to show that the block diagonal matrix on the right-hand side of \eqref{last_blk} has full column rank. But then it suffices to show that $\widebar R$ has full column rank, and this latter claim can be established by following a similar line of arguments as in the proof for setting \eqref{case_1}. This completes the proof.
\end{proof}

\begin{theorem}
Consider \eqref{line_form} with setting \eqref{case_1} or \eqref{case_2}. Then the function $\Psi$ defined in \eqref{R_form} is smooth on $\R^{1\times p}\backslash\{0\}$.
\end{theorem}

\begin{proof}
In view of \cite[Equation~5]{UM12} and recall that $p-m = 1$ (in both cases \eqref{case_1} and \eqref{case_2}), we only need to show that for any $R\in\R^{1\times p}\backslash\{0\}$, the linear map $G_R:\R^d\longrightarrow\R^q$ defined as $G_R(y):=(R\A(y))^\top$ is surjective, or equivalently, $G_R^*$ is injective. To proceed, fix any $R\in\R^{1\times p}\backslash\{0\}$ and consider any $z\in\R^q$ with $G_R^*(z) = 0$. Then we have for any $y\in\R^d$ that
\begin{equation*}
0 = \langle G_R^*(z), y\rangle = \langle z, G_R(y)\rangle =  \langle z, (R\A(y))^\top\rangle =  \langle \A^*(R^\top z^\top), y\rangle.
\end{equation*}
Thus we have $\A^*(R^\top z^\top) = 0$, which together with Lemma~\ref{lemma_LICQ} implies that $z = 0$. This completes the proof.
\end{proof}

Since $\Psi$ is smooth on $\R^{1\times p}\backslash\{0\}$, we can then apply standard gradient-based optimization methods to solving \eqref{R_prob} and obtain a stationary point of $\widetilde \Psi$ in \eqref{obj}. We next discuss how one can obtain a pseudo-projection from such a stationary point.

\subsection{Stationarity and improvement of function value}\label{sec42}
We discuss in this subsection how to obtain a pseudo-projection from a suitable stationary point $R^*$ of $\widetilde \Psi$ in \eqref{obj},
under mild assumptions. We start by showing how one can construct from $R^*$ a point satisfying the stationarity condition in Definition~\ref{def_ps}.

\begin{theorem}\label{thm1}
Consider \eqref{line_form} with setting \eqref{case_1} or \eqref{case_2}. Let $R^*$ be a stationary point of $\widetilde \Psi$ in \eqref{obj} and let $y^*$ achieve the infimum in \eqref{R_form} when $R = R^*$. Then
\begin{equation}\label{app_st}
0\in y^* - \widehat{y} + \A^*\left(N_{\{X:\; \rank(X)\le m\}}(\A(y^*))\right).
\end{equation}
If in addition $\rank(\A(y^*)) = m$, then we have
\begin{equation}\label{true_st}
0 \in  y^* - \widehat{y} + N_{\{y:\; \rank(\A(y)) \le m\}}(y^*).
\end{equation}
\end{theorem}
\begin{proof}
First, we define
\begin{equation}\label{Phi_fun}
\Phi(y, R): = \frac12\|y - \widehat{y}\|^2 + \delta_{\{(y, R):\; R\A(y) = 0\}}(y, R) + \delta_{\{R:\; RR^\top= 1\}}(R).
\end{equation}
Then we see from \eqref{obj} and the definition of $y^*$ that
\begin{equation}\label{re_w}
\widetilde\Psi(R^*) = \inf_{y}\Phi(y, R^*) = \Phi(y^*,R^*).
\end{equation}
On the other hand, we also have from the stationarity of $R^*$ that $0\in \partial \widetilde \Psi(R^*)= \partial\left(\Psi + \delta_{\{R:\; RR^\top= 1\}}\right)(R^*)$.
Using this, \eqref{re_w} and \cite[Theorem 10.13]{RW98}, we see further that
\begin{equation}\label{diff_inclusion}
(0, 0)\in\partial \Phi(y^*, R^*).
\end{equation}

Next, notice from Lemma~\ref{lemma_LICQ} that for any $U\in \R^{1\times q}$, $y\in \R^d$, $\lambda\in \R$ and $R\in \R^{1\times p}\backslash\{0\}$, the following implication holds:
\begin{equation*}
\mbox{$\A^*({R}^\top U) = 0$ and $U{\A(y)}^\top + \lambda R = 0$}\ \ \Longrightarrow\ \ \mbox{ $U = 0$ and $\lambda = 0$.}
\end{equation*}
This corresponds to the linear independence constraint qualification for the following optimization problem:
\[
\min_{y\in \R^d,R\in \R^{1\times p}}\ \  \frac12\|y - \widehat{y}\|^2 \ \ {\rm s.t.}\ \ R\A(y)=0\ \ {\rm and}\ \ RR^\top = 1.
\]
Using this, the definition of $\Phi$ in \eqref{Phi_fun}, \eqref{diff_inclusion} and \cite[Example 10.8]{RW98}, we deduce that there exist $V^*\in\R^{1\times q}$ and a scalar $\lambda^*$ such that the following Karash-Kuhn-Tucker conditions hold:
\begin{eqnarray}\label{kkt}
 y^* - \widehat{y} + {\A}^*({R^*}^\top V^*) = 0, & &\ \ \  V^*\left(\A(y^*)\right)^\top + \lambda^* R^*  = 0, \label{kkt_1}\\
 R^*{R^*}^\top - 1 = 0, & &\ \ \  R^*\A(y^*) = 0. \label{kkt_2}
\end{eqnarray}
Multiplying both sides of the second equation in \eqref{kkt_1} from the right by ${R^*}^\top$, and using the two equations in \eqref{kkt_2},  we obtain $\lambda^* = 0$ and thus
\begin{equation}\label{key}
 V^*\left(\A(y^*)\right)^\top = 0.
\end{equation}

We now show that
\begin{equation}\label{inclusion}
{R^*}^\top V^*\in N_{\{X:\; \rank(X)\le m\}}(\A(y^*)).
\end{equation}
To proceed, recall that $R^*\in \R^{1\times p}$, which implies $\rank({R^*}^\top V^*) \le 1$. According to \cite[Proposition 3.6]{Luke13}, in order to establish \eqref{inclusion}, it now remains to show that
\begin{equation}\label{addthis}
 [\ker({R^*}^\top V^*)]^{\perp}\cap[\ker(\A(y^*))]^{\perp} = \{0\}.
\end{equation}
To this end, take any $z\in[\ker({R^*}^\top V^*)]^{\perp}\cap[\ker(\A(y^*))]^{\perp}$. Then we have in particular that $z\in[\ker({R^*}^\top V^*)]^{\perp} = \Range({V^*}^\top R^*)$. This together with \eqref{key} implies that
$\A(y^*)z \in \A(y^*) \Range({V^*}^\top R^*) = \{0\}$.
Thus, we must have $z\in\ker\left(\A(y^*)\right)\cap\left[\ker\left(\A(y^*)\right)\right]^\perp$ and consequently $z = 0$. This proves \eqref{addthis} and hence \eqref{inclusion}. The desired relation \eqref{app_st} now follows immediately from \eqref{kkt_1} and \eqref{inclusion}.

Suppose in addition that $\rank(\A(y^*)) = m$. Then we have
\begin{equation*}
\begin{split}
&\A^*\left(N_{\{X:\; \rank(X)\le m\}}(\A(y^*))\right) \overset{\rm (a)}\subseteq \A^*\left(\widehat{N}_{\{X:\; \rank(X)\le m\}}(\A(y^*))\right)\\
& \overset{\rm (b)}\subseteq \widehat{N}_{\{y:\; \rank(\A(y)) \le m\}}(y^*)\overset{\rm (c)}\subseteq N_{\{y:\; \rank(\A(y)) \le m\}}(y^*),
\end{split}
\end{equation*}
where (a) follows from \cite[Proposition~3.6]{Luke13} and the fact that proximal normal vectors are regular normal vectors \cite[Example~6.16]{RW98}, (b) follows from \cite[Theorem~10.6]{RW98} and (c) follows from \cite[Proposition~6.5]{RW98}.
This together with \eqref{app_st} proves \eqref{true_st}. This completes the proof.
\end{proof}

We next show that if the stationary point $R^*$ of $\widetilde \Psi$ in \eqref{obj} is obtained via a gradient-based {\em descent} optimization method with a suitably chosen initial point, then the $y^*$ that attains the infimum in \eqref{R_form} will satisfy the condition on function value improvement in Definition~\ref{def_ps}.

\begin{theorem}\label{thm2}
Consider \eqref{line_form} with setting \eqref{case_1} or \eqref{case_2}. Let $y_b\in\R^d$ satisfy $\rank\left(\A(y_b)\right) \le m$ and let $R^0\in\R^{1\times p}\backslash\{0\}$ satisfy $R^0\A(y_b) = 0$. Then for any $\widetilde{R}\in \R^{1\times p}\backslash\{0\}$ with $\Psi(\widetilde{R}) \le \Psi(R^0)$, we have
\begin{equation}\label{fun_dec}
\|y_{\widetilde{R}} - \widehat{y}\| \le \|y_b - \widehat{y}\|,
\end{equation}
where $y_{\widetilde{R}}$ attains the infimum in \eqref{R_form} when $R =\widetilde{R}$.
\end{theorem}

\begin{proof}
First, we see from $R^0\A(y_b) = 0$ and the definition of $\Psi$ in \eqref{R_form} that
$\Psi(R^0) \le \frac12\|y_b - \widehat{y}\|^2$.
This together with the assumption $\Psi(\widetilde{R})\le \Psi(R^0)$ and the fact that $y_{\widetilde{R}}$ attains the infimum in \eqref{R_form} when $R =\widetilde{R}$ shows that
\begin{equation*}
 \frac12\|y_{\widetilde{R}} - \widehat{y}\|^2 =  \Psi(\widetilde{R}) \le \Psi(R^0) \le \frac12\|y_b - \widehat{y}\|^2.
\end{equation*}
This completes the proof.
\end{proof}

\begin{remark}[Obtaining pseudo-projection in cases \eqref{case_1} or \eqref{case_2}]\label{rem4_5}
Let $y_b\in \R^d$ satisfy $\rank(\A(y_b))\le m$ and let $R^0\in\R^{1\times p}\backslash\{0\}$ satisfy $R^0\A(y_b) = 0$. Then one can apply some standard gradient-based descent methods such as those implemented in SLRA \cite{MK14} for solving \eqref{R_prob} with $R^0$ as the initialization: these methods typically generate a sequence $\{R^k\}$ so that any accumulation point, say $R^*$, is stationary for $\widetilde\Psi$ in \eqref{obj} and satisfies $\Psi(R^*)\le\Psi(R^0)$. Suppose $y_{R^*}$ achieves the infimum in \eqref{R_form} when $R =R^*$. Then we know from \eqref{true_st} in Theorem~\ref{thm1} and \eqref{fun_dec} in Theorem~\ref{thm2} that if $\rank(\A(y_{R^*})) = m$ holds, then
$y_{R^*}\in\P^s_{\rank(\A(y))\le m}(\widehat{y}; y_b)$.
\end{remark}

\subsection{Conjecture related to Theorem~\ref{thm1}}
In this subsection, we revisit the assumption $\rank(\A(y^*)) = m$ in Theorem~\ref{thm1}. We would like to understand how likely such a condition is fulfilled by the $y^*$ that achieves the infimum in \eqref{R_form}, with $R = R^*$ being a stationary point of $\widetilde \Psi$ in \eqref{obj}. Notice that if $R^*$ is indeed an optimal solution of $\widetilde \Psi$, such a $y^*$ is an optimal solution of \eqref{line_form}. Thus, we will first study whether $\rank(\A(y^*)) = m$ when $y^*$ is an optimal solution of \eqref{line_form}. Specifically, we make the following conjecture:
\begin{conjecture}\label{guess}
Let $s$ be a positive integer. Suppose that $\widehat{y}\in \R^n$ satisfies the condition $\rank(\H_{s+1}\left(\widehat{y})\right) = s+1$ and let $y^*$ solve the following optimization problem:
\begin{equation}\label{citeme}
\min_{y\in\R^n}\ \ \frac12\|y - \widehat{y}\|^2 \ \ {\rm s.t.}\ \ \rank\left(\H_{s+1}(y)\right) \le s.
\end{equation}
Then we have $\rank\left(\H_{s+1}(y^*)\right) = s$.
\end{conjecture}

We do not know whether Conjecture~\ref{guess} holds true for all positive numbers $s$. However, we are able to prove that it holds true when $s = 1$.
\begin{proposition}
  Conjecture~\ref{guess} holds true when $s = 1$.
\end{proposition}
\begin{proof}
Since $s = 1$, we only need to show that there exists $\widebar{y}\in\R^n$ with $\rank\left(\H_2(\widebar{y})\right) = 1$ and $\|\widebar{y}- \widehat{y}\|^2 < \|\widehat{y}\|^2$. First of all, since $\rank(\H_2(\widehat y)) = 2$, we must have $n\ge 3$. We consider two cases:
\vspace{-1mm}
\begin{equation*}
{\rm (i)}\ \widehat{y}(1) \neq 0\ \ {\rm or}\ \  \widehat{y}(n) \neq 0; \ \ \ \  {\rm (ii)}\ \widehat{y}(1) = 0 \ \ {\rm and}\ \ \widehat{y}(n) = 0.
\end{equation*}
For case (i), we let $\widebar{y} = [\widehat{y}(1)\ 0 \cdots 0]^\top $ when $\widehat{y}(1) \neq 0$, and $\widebar{y} = [0\cdots 0\ \widehat{y}(n)]^\top $ when $\widehat{y}(n) \neq 0$. Then $\rank\left(\H_2(\widebar{y})\right) = 1$ and
\vspace{-2.5mm}
\begin{equation*}
\|\widebar{y} - \widehat{y}\|^2 = \sum_{i=2}^n\widehat{y}^2(i) < \|\widehat{y}\|^2 \ \  {\rm or}\ \  \|\widebar{y} - \widehat{y}\|^2 = \sum_{i=1}^{n-1}\widehat{y}^2(i) < \|\widehat{y}\|^2.
\end{equation*}

\vspace{-1mm}
\noindent
Now we consider case (ii). Notice that there exists at least one nonzero element in $\{\widehat{y}(2),\cdots,\widehat{y}(n-1)\}$ because $\rank\left(\H_2(\widehat{y})\right) = 2$. Hence, there are at most $n - 2$ distinct real roots for the polynomial equation $\sum_{i=2}^{n-1}\widehat{y}(i)(z)^{i-1} = 0$. Let $\bar{z}\neq 0$ be a real number different from these roots. Then we have $\sum_{i=0}^{n-1}(\widebar{z})^{2i} > 0$. Let
\vspace{-2mm}
\begin{equation*}
\widebar{c} = \sum_{i=2}^{n-1}\widehat{y}(i)(\bar{z})^{i-1}\Big{/}\sum_{i=0}^{n-1}(\bar{z})^{2i} \ \ {\rm and}\ \ \widebar{y} = [\widebar{c}\ \ \widebar{c}\widebar{z}\ \cdots\ \widebar{c}\widebar{z}^{n-1}]^\top.
\end{equation*}
Then $\widebar{c}\neq 0$ and $\rank\left(\H_2(\widebar{y})\right) = 1$. Consequently,
\begin{equation*}
\begin{split}
&\|\widebar{y}-\widehat{y}\|^2 - \|\widehat{y}\|^2 = \|\widebar y\|^2 - 2{\widebar y}^\top{\widehat y}= \widebar{c}^2\sum_{i=0}^{n-1}(\bar{z})^{2i}  - 2\widebar{c}\sum_{i=2}^{n-1}\widehat{y}(i)(\bar{z})^{i-1} =  - \widebar{c}^2\sum_{i=0}^{n-1}(\bar{z})^{2i}  < 0.
\end{split}
\end{equation*}
This completes the proof.
\end{proof}

\section{Numerical experiments}\label{sec_numerical}
In this section, we will conduct numerical experiments for our hybrid penalty method, {\em i.e.}, Algorithm~\ref{alg_hyb}. All numerical experiments are performed in Matlab R2019a on a 64-bit PC with 3.8 GHz Intel Core i5 Quad-Core and 8GB of DDR4 RAM.

We consider the following problem with two rank constraints:
\vspace{-2mm}
\begin{eqnarray}
  \min_{y_1\in\R^n,y_2\in\R^n} & &\ \ \frac12\|y_1 - \widebar{y}_1\|_{W}^2 + \frac12\|y_2 - \widebar{y}_2\|_{W}^2 \nn\\
  {\rm s.t.} & &\ \  \rank\left(\H_{n_1 + n_c +1}(y_1)\right) \le n_1 + n_c, \nn\\
   & & \ \ \rank\left(\H_{n_2 + n_c + 1}(y_2)\right) \le n_2 + n_c, \label{example}\\
   & & \ \  \rank\left(\left[\H_{n_1 + n_2 + n_c + 1}(y_1) \ \ \H_{n_1 + n_2 + n_c + 1}(y_2)\right]\right) \le n_1 + n_2 + n_c,\nn
\end{eqnarray}
where  $\|y\|_W:= \sqrt{y^\top W y}$, $W$ is the $n\times n$ diagonal matrix so that $W(i,i)$ equals $1$ when $i$ is odd, and equals $10$ when $i$ is even, $n_1$, $n_2$ and $n_c$ are given positive integers, and $\widebar{y}_1\in\R^n$ and $\widebar{y}_2\in\R^n$ are known noisy  signals.

Let \textbf{HB\_1}, \textbf{HB\_2} and \textbf{HB\_3} represent the three hybrid penalty methods which solve \eqref{example} by Algorithm~\ref{alg_hyb} via the reformulation \eqref{uni_form} with \textbf{Variant I}, \textbf{Variant II} and \textbf{Variant III} discussed in Section~\ref{sdcam_sec} respectively. Let \textbf{AP} represent the alternating pseudo-projection algorithm \eqref{alg_AP} applied directly to the sets $\Omega_1$ and $\Omega_2$ defined in \eqref{omega_ap}, constructed based on the data from \eqref{example}.

\vspace{2mm}

\noindent
{\bf Data generation:} We set  $n = 50$ and consider two 3-tuples $(n_1, n_2, n_c) = (2, 2, 2)$ and $(n_1, n_2, n_c) = (2, 6, 4)$. For each 3-tuple, we first randomly generate two signals $y_1$ and $y_2$ from two marginally stable linear time-invariant systems of order at most $n_1 + n_c$ and $n_2 + n_c$ respectively, which have $n_c$ common poles. Then we let $\widebar{y}_1 =  y_1 + \sigma\cdot W^{-1/2}\xi_1$ and $\widebar{y}_2 = y_2 + \sigma\cdot W^{-1/2}\xi_2$, where $\sigma = 0.1$ is the noise factor, and  $\xi_1$ and $\xi_2$  are random vectors with i.i.d. standard Gaussian entries.

\vspace{2mm}

\noindent
{\bf HB\_1, HB\_2 and HB\_3:}  In Algorithm~\ref{alg_npg}, we set $L_{\max} = 10^8$, $L_{\min} = 10^{-8}$, $\tau = 2$, $c = 10^{-4}$, $M = 4$,  $L_0^0 = 1$ and for $l \ge 1$,
\vspace{-2mm}
\begin{equation*}
L_l^0 = \max\left\{\min\left\{ \frac{{(y^l - y^{l-1})}^\top\left(\nabla h(y^{l}) - \nabla h(y^{l-1})\right)}{\|y^{l} - y^{l-1}\|^2}, L_{\max}\right\}, L_{\min}\right\}.
\end{equation*}

\vspace{-2mm}
\noindent
All pseudo-projection subproblems that arise are approximately solved by calling SLRA \cite{MK14} with default setting (except that the $R^0$ is specified as in Remark~\ref{rem4_5}). We terminate Algorithm~\ref{alg_npg} when the number of iterations exceeds 10$^8$ or
\vspace{-2mm}
\begin{equation*}
\frac{\|y^{l} - y^{l-1}\|}{\max\left\{\|y^{l}\|, 1\right\}}  < \epsilon_t/\widebar{L}_{l-1} \ \ \ {\rm or}\ \ \ \frac{\left|F_{\lambda_t}(y^{l} )- F_{\lambda_t}(y^{l-1})\right|}{\max\left\{ |F_{\lambda_t}(y^{l})|, 1\right\}} < 10^{-10}.
\end{equation*}

\vspace{-1mm}
For the penalty method in Algorithm~\ref{alg_hyb}, we set  $y^{{\rm feas}} = 0$,  $\lambda_t = \lambda_{t-1}/5$ with initial $\lambda_0 = 0.1$, $\widebar{\lambda} = 10^{-4}$ and $\epsilon_t = \max\left\{\epsilon_{t-1}/1.5, 10^{-6}\right\}$ with initial $\epsilon_0 = 10^{-5}$. Let  $\widebar{y} = vec\left(\widebar{y}_1\ \widebar{y}_2\right)$. We set the initial point $y^0$ for {\bf HB\_1} and {\bf HB\_2} as a pseudo-projection of $\widebar{y}$ onto $\Omega_1$ and $\Omega_2$ respectively, obtained by calling SLRA in \cite{MK14} with default setting (the reference point is the origin). For {\bf HB\_3}, we set $y^0 = \widebar{y}$.

For the post-processing method in Algorithm~\ref{alg_hyb}, we also call SLRA in \cite{MK14} with default settings to approximately compute a pseudo-projection (except that the $R^0$ is specified as in Remark~\ref{rem4_5}), and terminate it when the number of iterations exceeds 10$^5$ or
\begin{equation*}
\frac{\max\{\|x^{t} - x^{t-1}\|, \|z^{t} - z^{t-1}\|\}}{\max\{\|x^{t-1}\|, \|z^{t-1}\|,1\}} < 10^{-10}.
\end{equation*}
We output $z^t$ as the approximate solution.

\noindent
{\bf AP:} In this method, we start at $\widebar{y} = vec\left(\widebar{y}_1\ \widebar{y}_2\right)$ and call SLRA in \cite{MK14} with default setting (except that the $R^0$ is specified as in Remark~\ref{rem4_5}) to approximately compute a pseudo-projection onto $\Omega_1$ and $\Omega_2$ defined in \eqref{omega_ap} (the initial reference points are the origin). We also output $z^t$ as the approximate solution.

\vspace{2mm}

\noindent
{\bf Numerical results:}
In Figure~\ref{all_fval}, we compare the four methods {\bf AP}, \textbf{HB\_1}, \textbf{HB\_2} and \textbf{HB\_3}  in terms of  terminating function values over 100 random instances for $(n_1, n_2, n_c) = (2, 2, 2)$ and over 30 random instances for $(n_1, n_2, n_c) = (2, 6, 4)$. \footnote{For each 3-tuple, we first generate $y_1$ and $y_2$ as described above. For these two fixed signals, we generate 100 (and, resp., 30) random noisy signals  $\bar{y}_1$ and $\bar{y}_2$ and solve the corresponding instances.} One can see that while the three hybrid penalty methods \textbf{HB\_1}, \textbf{HB\_2} and \textbf{HB\_3} have comparable performance, they always outperform {\bf AP}.

In Figure~\ref{hybrid_only}, we compare the three hybrid penalty methods \textbf{HB\_1}, \textbf{HB\_2} and \textbf{HB\_3} in terms of constraint violation (before and after post-processing) and CPU time over 30 random instances for $(n_1, n_2, n_c) = (2, 6, 4)$. We measure constraint violation by ${\rm log}_{10}(vio)$, with \emph{vio} given by
{\small{
\begin{equation*}  \max\left\{\frac{\dist(\H_{m_1 +1}(y_1^*), \Xi_{m_1})}{\|\H_{m_1 +1}(y_1^*)\|_2}, \frac{\dist(\H_{m_2 +1}(y_2^*), \Xi_{m_2})}{\|\H_{m_2 +1}(y_2^*)\|_2}, \frac{\dist([\H_{m+1}(y_1^*),\H_{m+1}(y_2^*)], \Xi_{m })}{\|[H_{m+1}(y_1^*),\H_{m+1}(y_2^*)]\|_2}\right\},
\end{equation*}}}where $y_1^*$ and $y_2^*$ are computed solutions, $m_1 = n_1 + n_c$, $m_2 = n_2 + n_c$, $m = n_1 + n_2 +n_c$ and $\Xi_{s} := \{Y: \rank(Y) \le s\}$.
One can see that the post-processing scheme significantly reduces constraint violation. On the other hand, \textbf{HB\_2} is faster than \textbf{HB\_1} and \textbf{HB\_3}.

\begin{figure}[h]
  \centering
\begin{subfigure}{.5\textwidth}
  \includegraphics[width=1.0\linewidth]{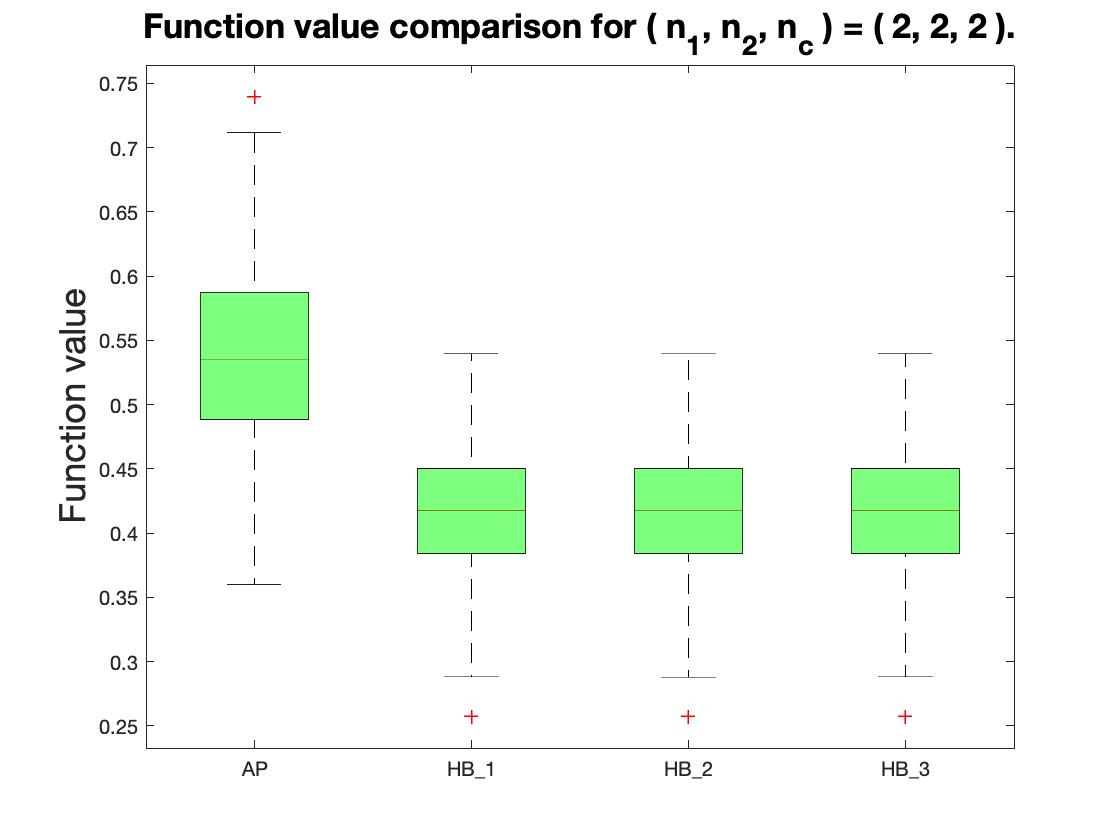}
\end{subfigure}%
\begin{subfigure}{.5\textwidth}
  \includegraphics[width=1.0\linewidth]{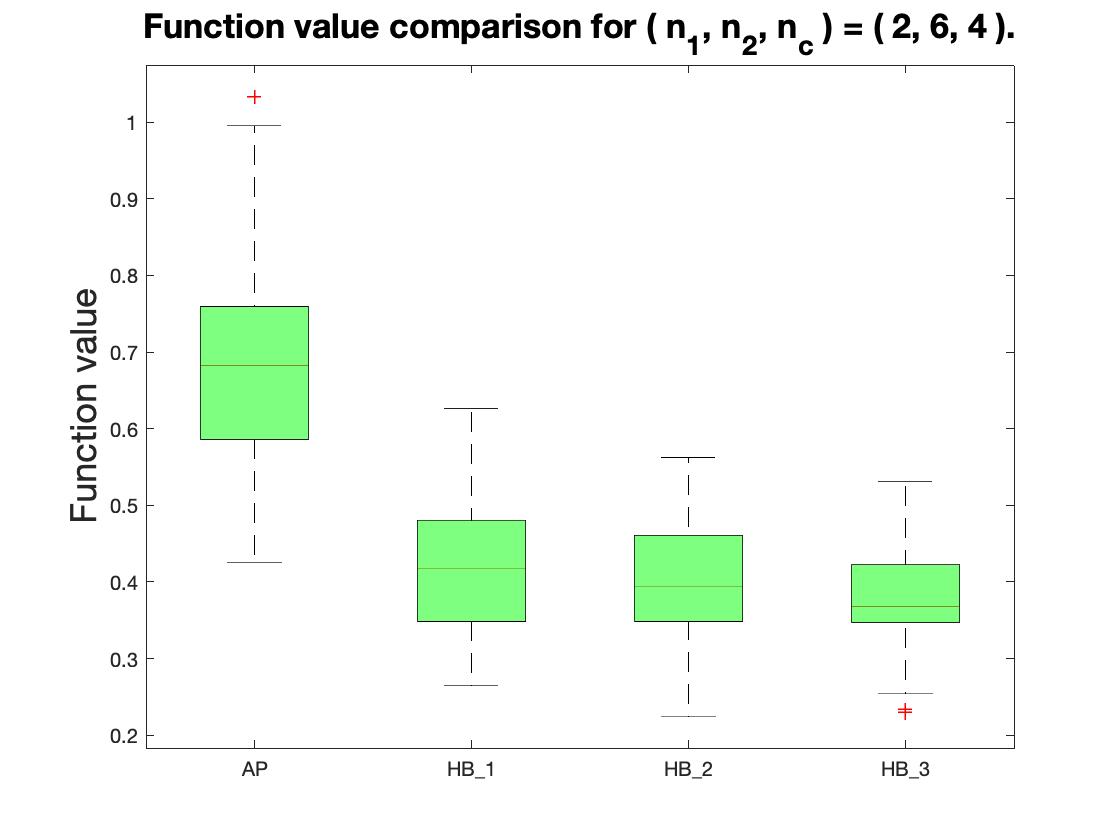}
  \end{subfigure}
\caption{Comparing terminating function values among {\bf AP}, {\bf HB\_1}, {\bf HB\_2} and {\bf HB\_3}.}
\label{all_fval}
\end{figure}

\begin{figure}[h]
  \centering
\begin{subfigure}{.5\textwidth}
  \includegraphics[width=1.0\linewidth]{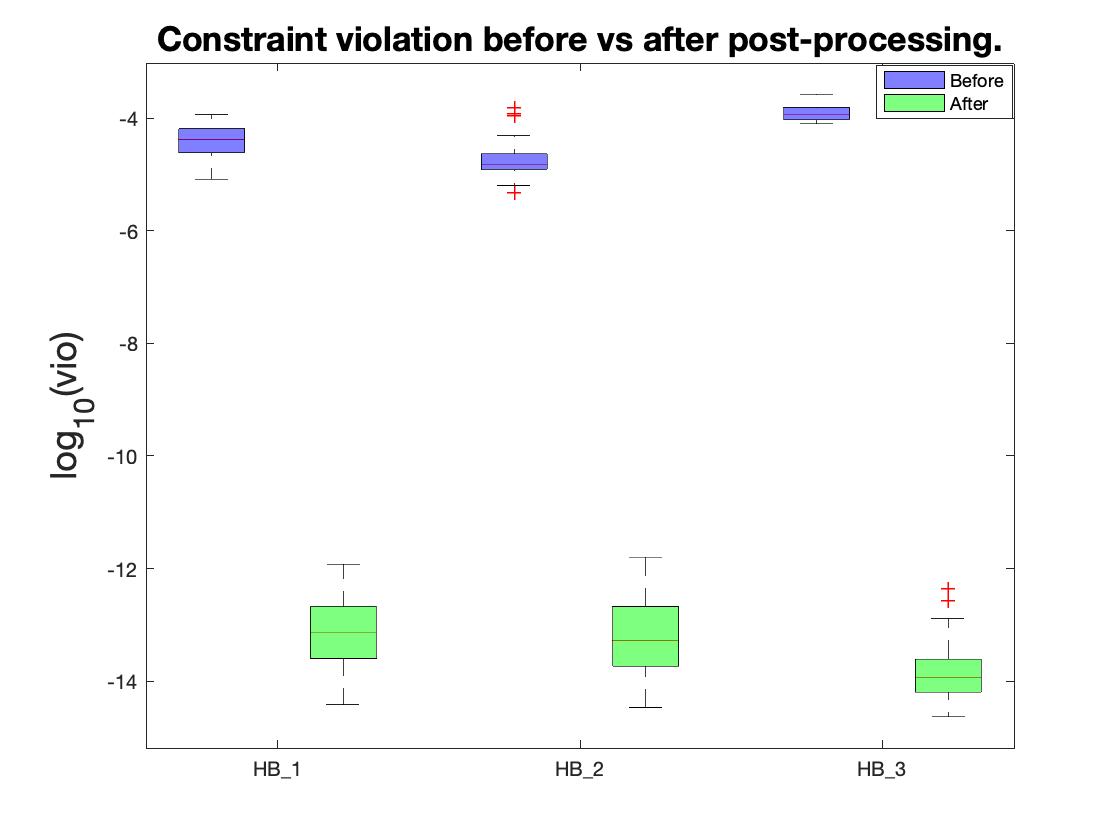}
\end{subfigure}%
\begin{subfigure}{.5\textwidth}
  \centering
  \includegraphics[width=1.0\linewidth]{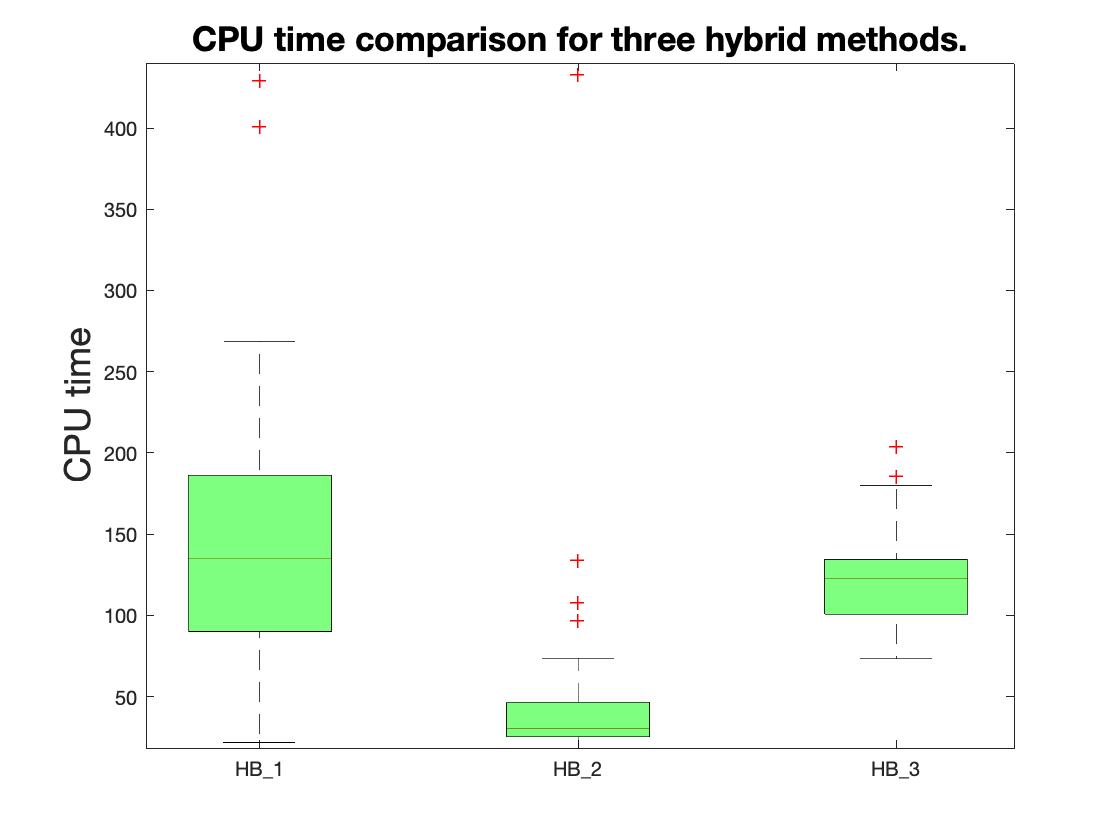}
\end{subfigure}
\caption{Comparing constraint violations and CPU times (in seconds) among {\bf HB\_1}, {\bf HB\_2} and {\bf HB\_3}.}
\label{hybrid_only}
\end{figure}

\section{Concluding remarks}\label{sec_conclusion}
In this paper, we propose a hybrid penalty method for solving \eqref{multi_model}. The hybrid penalty method consists of two parts: a penalty scheme which makes use of a special penalty function as in \cite{LPT18}, and a post-processing method for reducing constraint violation. Both the penalty subproblems and the subproblems in the post-processing method involve the new concept of pseudo-projections: we discussed in Section~\ref{sec_pseudo} in detail how pseudo-projections can be computed efficiently by some existing software such as \cite{MK14}, under mild assumptions.

There are several open questions related to pseudo-projection computation. For instance, we still do not know how likely the condition $\rank(\A(y^*)) = m$ holds for the $y^*$ that achieves the infimum in \eqref{R_form} (with $R = R^*$ being a stationary point of $\widetilde \Psi$ in \eqref{obj}). \footnote{In the numerical experiments in Section~\ref{sec_numerical}, the condition $\rank(\A(y^*)) = m$ almost never fails for the solution $y^*$ returned by SLRA: For over 99.9\% of our calls to SLRA, the $m$th singular value of $\A(y^*)$ is significantly larger than its next singular value.} Even assuming $y^*$ is a solution of \eqref{citeme}, we can only establish $\rank(\H_{s+1}(y^*)) = s$ when $s=1$. The case for $s > 1$ is still open.

\appendix
\section{Proof of Theorem~\ref{thm_ap}}
Before proving Theorem~\ref{thm_ap}, we first state two auxiliary lemmas without proofs. The proof of Lemma~\ref{CQ_eq} can be found in the first paragraph in the proof of \cite[Theorem 5.16]{LLM07}, and Lemma~\ref{lemma_be} follows from Theorem~\ref{prox_proof} and the same argument as in the proof of \cite[Theorem~5.16]{LLM07}.

\begin{lemma}\label{CQ_eq}
Let $\Omega_1$ and $\Omega_2$ be defined as in \eqref{omega_ap}, $\widebar{y}\in\Omega_1\cap\Omega_2$ and define
\begin{equation}\label{c_def}
\widebar{c}: = \max\left\{\langle u, v\rangle: u\in N_{\Omega_1}(\widebar{y})\cap B, \ \ \  v\in -N_{\Omega_2}(\widebar{y})\cap B\right\},
\end{equation}
where $B$ is the closed unit ball. Then $N_{\Omega_1}(\widebar{y})\cap -N_{\Omega_2}(\widebar{y}) = \{0\}$ if and only if $\widebar c < 1$.
\end{lemma}

\begin{lemma}\label{lemma_be}
Let $\Omega_1$ and $\Omega_2$ be defined as in \eqref{omega_ap}. Suppose that there exists some $\widebar{y}\in\Omega_1\cap\Omega_2$ such that $\rank(\L(\widebar{y})) = r$ and $N_{\Omega_1}(\widebar{y})\cap -N_{\Omega_2}(\widebar{y})  = \{0\}$. Let $\widebar{c}$ be defined as in \eqref{c_def}. Then for any $c\in(\widebar{c}, 1)$, there exist some $\epsilon > 0$ and $\delta\in[0, \frac{1-c}2)$ such that
\begin{equation}\label{1_ineq}
\left.
\begin{array}{ll}
x\in\Omega_1\cap B_{\epsilon}(\widebar{y}),\ \ u\in N_{\Omega_1}(x)\cap B \\
z\in\Omega_2\cap B_{\epsilon}(\widebar{y}),\ \ v\in -N_{\Omega_2}(z)\cap B
\end{array}
\right\}\Longrightarrow \langle u, v\rangle \le c,
\end{equation}

\vspace{-6mm}
\begin{equation}\label{2_ineq}
\left.
\begin{array}{ll}
x, z\in\Omega_2\cap B_{\epsilon}(\widebar{y})\\
v\in N_{\Omega_2}(z)\cap B
\end{array}
\right\}\Longrightarrow \langle v, x - z\rangle \le \delta\|x - z\|,
\end{equation}
where $B_{\epsilon}(\widebar{y})$ is the closed ball with centre $\widebar{y}$ and radius $\epsilon$, and $B$ is the closed unit ball.
\end{lemma}

We now prove Theorem~\ref{thm_ap}. The proof follows the same line of arguments as in \cite[Theorem~5.2]{LLM07}.

\begin{proof}
Fix any $c\in (\widebar{c}, 1)$ with $\widebar{c}$ defined as in \eqref{c_def}, and let $\delta$ and $\epsilon$ be given as in Lemma~\ref{lemma_be}.
We first claim that
\begin{equation}\label{claim}
\left.
\begin{array}{ll}
\|z^{t+1} - \widebar{y}\| \le \frac{\epsilon}2\\
\|z^{t+1} - x^t\| \le \frac{\epsilon}2
\end{array}
\right\}\Longrightarrow \|x^{t+1} - z^{t+1}\| \le c_0\|x^t - z^{t+1}\|,
\end{equation}
where $c_0: = c + 2\delta$. To prove this, note from \eqref{aap_iter} and Definition~\ref{def_ps} that
\begin{eqnarray}
x^t - z^{t+1}\in N_{\Omega_1}(z^{t+1}),&&  \ \ z^{t+ 1} - x^{t+1} \in N_{\Omega_2}(x^{t+1}),\label{eq_1}\\
\|z^{t+1} - x^t\| \le \|z^t - x^t\|,&&\ \ \|x^{t+1} - z^{t+1}\| \le \|x^t - z^{t+1}\|. \label{eq_2}
\end{eqnarray}
If $\|x^{t+1} - z^{t+1}\| = 0$ or $\|x^t - z^{t+1}\| = 0$, we then see from the second inequality in \eqref{eq_2} that \eqref{claim} holds trivially. Now we assume that $\|x^{t+1} - z^{t+1}\| \neq 0$ and $\|x^t - z^{t+1}\| \neq 0$. We first notice from \eqref{eq_2}, $\|z^{t+1} - \widebar{y}\| \le \frac{\epsilon}2$ and $\|z^{t+1} - x^t\| \le \frac{\epsilon}2$ that
\begin{eqnarray}
\|x^{t+1} - \widebar{y}\| &\le& \|x^{t+1} - z^{t+1}\| + \|z^{t+1} - \widebar{y}\| \le \|z^{t+1} - x^t\| + \|z^{t+1} - \widebar{y}\| \le \epsilon,\label{ieq_1}\\
\|x^t - x^{t+1}\| &\le& \|x^t - z^{t+1}\| + \|z^{t+1} - x^{t+1}\| \le 2\|x^t - z^{t+1}\|, \label{ieq_2}\\
\|x^t - \widebar{y}\| &\le& \|x^t - z^{t+1}\| + \|z^{t+1} - \widebar{y}\| \le \epsilon. \label{ieq_3}
\end{eqnarray}
Using \eqref{eq_1}, \eqref{ieq_1}  and $\|z^{t+1} - \bar y\| \le \frac{\epsilon}2$, we obtain further that
\begin{eqnarray}
\textstyle\frac{x^t - z^{t+1}}{\|x^t - z^{t+1}\|} \in N_{\Omega_1}(z^{t+1})\cap B && \ {\rm with}\ \ z^{t+1}\in\Omega_1\cap B_{\epsilon}(\widebar{y}) \label{ee_1}\\
\textstyle\frac{x^{t+1} - z^{t+1}}{\|x^{t+1} - z^{t+1}\|}\in - N_{\Omega_2}(x^{t+1})\cap B && \ {\rm with}\ \ x^{t+1}\in\Omega_2\cap B_{\epsilon}(\widebar{y}).\label{ee_2}
\end{eqnarray}
Here, $B$ represents the closed unit ball and $B_{\epsilon}(\widebar{y})$ represents the closed ball with center $\widebar{y}$ and radius $\epsilon$. Furthermore, we see from \eqref{1_ineq}, \eqref{ee_1} and \eqref{ee_2} that
\begin{equation}\label{plus_1}
\langle x^t - z^{t+1}, x^{t+1} - z^{t+1}\rangle \le c\|x^t - z^{t+1}\|\|x^{t+1} - z^{t+1}\|.
\end{equation}
On the other hand, in view of \eqref{ieq_1}, \eqref{ieq_3} and \eqref{ee_2}, we can apply \eqref{2_ineq} with $x = x^t$, $z= x^{t+1}$ and $v = \frac{z^{t+1} - x^{t+1}}{\|z^{t+1} - x^{t+1}\|}$ to obtain
\begin{equation}\label{plus_2}
\begin{split}
&\langle x^t - x^{t+1}, z^{t+1} - x^{t+1}\rangle \le \delta\|x^t - x^{t+1}\|\|z^{t+1} - x^{t+1}\| \\
&\le 2\delta\|x^t - z^{t+1}\|\|z^{t+1} - x^{t+1}\|,
\end{split}
\end{equation}
where the second inequality follows from \eqref{ieq_2}. Adding \eqref{plus_1} and \eqref{plus_2}, we obtain
\begin{equation*}
\|x^{t+1} - z^{t+1}\| \le (c + 2\delta)\|x^t - z^{t+1}\|  = c_0\|x^t - z^{t+1}\|,
\end{equation*}
which proves \eqref{claim}.

Note from $c_0= c + 2\delta$  with $c\in(\bar{c}, 1)$ and $\delta\in [0,\frac{1-c}2)$ that $c_0\in(0, 1)$. Choose initial points $x^0$ and $z^0$ such that $\gamma := \|x^0 - \widebar{y}\| + \|z^0 - x^0\| < \frac{(1-c_0)\epsilon}4$. Next, we prove the following inequalities by induction:
\vspace{-2mm}
\begin{eqnarray}
\|z^{t+1} -x^t\| &\le & \gamma {c_0}^t < \textstyle\frac{\epsilon}2, \label{de_1}\\
\|z^{t+1} - \widebar{y}\| &\le& \textstyle2\gamma\frac{1- {c_0}^{t+1}}{1 - c_0} < \frac{\epsilon}2, \label{de_2}\\
\|x^{t+1} - z^{t+1}\| &\le& \gamma {c_0}^{t + 1}. \label{de_3}
\end{eqnarray}

\vspace{-1mm}
\noindent
First, we prove that the above three inequalities hold for $t = 0$. Note from $c_0\in(0, 1)$, the $z$-update in \eqref{aap_iter} and the definition of $\gamma$ that
\begin{equation*}
\textstyle\|z^1 - x^0\| \le \|z^0 - x^0\| \le \gamma < \frac{\epsilon}2\ \ {\rm and}\ \ \|z^1 - \widebar{y}\| \le \|z^1 - x^0\| + \|x^0 - \widebar{y}\| \le  2\gamma < \frac{\epsilon}2.
\end{equation*}
which proves \eqref{de_1} and \eqref{de_2} for $t = 0$.
Then we see from $\|z^1 - x^0\| < \frac{\epsilon}2$, $\|z^1 - \widebar{y}\| < \frac{\epsilon}2$ and \eqref{claim} that
\vspace{-3mm}
\begin{equation*}
\|x^1 - z^1\| \le c_0\|x^0 - z^1\| \le \gamma c_0,
\end{equation*}
which proves \eqref{de_3} for $t = 0$. To prove by induction, we assume that \eqref{de_1}, \eqref{de_2} and \eqref{de_3} hold for some $t\ge 0$. We know from the $z$-update, \eqref{de_1}  and \eqref{de_3} that
\begin{equation*}
\|z^{t+2} - x^{t+1}\| \le \|z^{t+1} - x^{t+1}\| \le \gamma c_0^{t+1} < \textstyle\frac{\epsilon}2.
\end{equation*}
This together with \eqref{de_2} and \eqref{de_3} implies
\begin{eqnarray*}
\|z^{t+2} - \widebar{y}\| &\le& \|z^{t+2} - x^{t+1}\| + \|x^{t+1} - z^{t+1}\| + \|z^{t+1} - \widebar{y}\|\\
& \le& \textstyle\gamma c_0^{t+1} + \gamma c_0^{t+1} + 2\gamma\frac{1-c_0^{t+1}}{1-c_0} = 2\gamma\frac{1 - c_0^{t+2}}{1-c_0} < \frac{2\gamma}{1-c_0} < \frac{\epsilon}2.
\end{eqnarray*}
We then see from $\|z^{t+2} - x^{t+1}\| < \frac{\epsilon}2$, $\|z^{t+2} - \widebar{y}\|  < \frac{\epsilon}2$ and \eqref{claim} that
\vspace{-2mm}
\begin{equation*}
\|x^{t+2} - z^{t+2}\| \le c_0\|x^{t+1} - z^{t+2}\| \le \gamma c_0^{t+2}.
\end{equation*}
Thus, we proved \eqref{de_1}, \eqref{de_2} and \eqref{de_3} for $t+1$. This completes the induction.

Now we prove that the sequence $\{z^0, x^0, z^1, x^1\cdots\}$ is a Cauchy sequence. For any $t$ and $k > s \ge t$, we know from \eqref{de_1} and \eqref{de_3} that
\begin{eqnarray*}
\|z^k - z^s\| &\le& \textstyle\sum_{j=s}^{k-1}\left(\|z^{j+1} - x^j\| + \|x^j - z^j\|\right) \le 2\gamma\left(c_0^s + c_0^{s+1} + \cdots + c_0^{k-1} \right) \le \frac{2\gamma c_0^{t}}{1-c_0},\\
\|x^k - x^s\| &\le& \textstyle\sum_{j=s}^{k-1}\left(\|x^{j+1} - z^{j+1}\|\! +\!\|z^{j+1} - x^j\|\right) \!\le\! \gamma\sum_{j=s}^{k-1}c_0^{j+1} \!+\! \gamma\sum_{j=s}^{k-1}c_0^{j} \!\le\! \frac{\gamma c_0^t(1+c_0)}{1-c_0}.
\end{eqnarray*}
Furthermore, by using \eqref{de_3}, we  have
\begin{eqnarray*}
\|z^k - x^s\| &\le& \|z^k - z^s\| + \|z^s - x^s\| \le \textstyle\frac{2\gamma c_0^{t}}{1-c_0} + \gamma c_0^t,\\
\|x^k - z^s\| &\le& \|x^k - x^s\| + \|x^s - z^s\| \le  \textstyle\frac{\gamma c_0^t(1+c_0)}{1-c_0} + \gamma c_0^t.
\end{eqnarray*}
These prove that the sequence $\{z^0, x^0, z^1, x^1\cdots\}$ is a Cauchy sequence. Therefore, it converges to some $y^*\in\Omega_1\cap\Omega_2$ and we have for any $t$ that
\begin{equation*}
\|z^t - y^*\| \le \textstyle\frac{2\gamma c_0^{t}}{1-c_0}\ \ \ \ {\rm and}\ \ \ \ \|x^t - y^*\| \le \textstyle\frac{\gamma c_0^t(1+c_0)}{1-c_0}.
\end{equation*}
Thus the sequence $\{z^0, x^0, z^1, x^1\cdots\}$ converges $R$-linearly. This completes the proof.
\end{proof}

\end{document}